\documentclass[12pt]{amsart}
\usepackage{url}
\usepackage{amsmath}
\usepackage[makeroom]{cancel}
\usepackage{graphicx}
\usepackage[all]{xy}
\usepackage[mathscr]{eucal}
\usepackage{amsmath,amssymb,amsfonts}
\newtheorem{theorem}{Theorem}[section]
\newtheorem{lemma}[theorem]{Lemma}

\newtheorem{example}[theorem]{Example}

\newtheorem{proposition}[theorem]{Proposition}

\usepackage{stackengine}
\usepackage{xcolor}
\usepackage[all]{xy}
\newcommand{\dom}{\mathbf{d}}
\newcommand{\ran}{\mathbf{r}}

\title[Inverse semigroups]{Introduction to inverse semigroups}
\author{Mark~V.~Lawson}
\address{Department of Mathematics and the
Maxwell Institute for Mathematical Sciences\\
Heriot-Watt University\\
Riccarton\\
Edinburgh~EH14~4AS\\
Scotland\\
\texttt{m.v.lawson@hw.ac.uk}} 

\begin{document}
\begin{abstract}
This is an account of the theory of inverse semigroups, assuming only that the reader knows
the basics of semigroup theory.
\end{abstract}
\maketitle

\section{Introduction: a little history}\setcounter{theorem}{0}

Inverse semigroups\footnote{The early development of inverse semigroup theory ocurred at a time
of acute tensions between East and West. See the book by Christopher Hollings \cite{H} for an in-depth analysis.}    were introduced in the 1950's
by Wagner \cite{Wagner1952} in the Soviet Union\footnote{The book \cite{HL} contains more information on Wagner and his approach to mathematics.
See also Wagner's obituary by Boris Schein \cite{Schein1981}.},
by Preston \cite{Preston1954} in the UK\footnote{The connection between Preston's introduction of what he termed `inverse semi-groups' 
and the theory of pseudogroups of transformations is the result of a conversation I had with Preston in Australia in the late 1980's.},
and by Ehresmann \cite{Ehresmann1957} in France\footnote{It is hard to single out a single paper of Ehresmann as the starting point for his approach to inverse semigroups --- which was via a class of ordered groupoids --- since he stayed faithful to his interest in pseudogroups as the way in which local structures were defined. 
In any event, Ehresmann's papers in this area reflect what might be called `an evolving corpus of work'. See \cite{E}.}
as algebraic analogues of pseudogroups of transformations. 
Inverse semigroups can be viewed as generalizations of groups.
Whereas group theory is based on the notion of a symmetry 
--- that is, a structure-preserving bijection ---
inverse semigroup theory is based on the notion of a partial symmetry
--- that is, a structure-preserving partial bijection.
The passage from bijections to partial bijections introduces a host of new algebraic phenomena.

In writing this chapter, I have assumed the reader is familiar with the basics of semigroup theory such as could be gleaned from 
the first few chapters of Howie \cite{Howie}.\footnote{On page~128, Howie quotes Clifford and Preston as saying that inverse semigroups were the most promising class of semigroups 
for future study. This prediction has turned out to be correct.}
Here are a few things you should know and which I take for granted.
The multiplication in a semigroup will usually be denoted by $\cdot$ or by concatenation.
If $S$ is any semigroup then $S^{0}$ is the semigroup $S$ with an adjoined zero,
and $S^{1}$ is the semigroup $S$ with an adjoined identity.
If $A$ and $B$ are any subsets of a semigroup $S$ then $AB$ is the set of all
products where $a \in A$ and $b \in B$. 
The study of congruences in semigroups is unavoidable
since there are not, in general, the analogues of normal subgroups as in groups or ideals as in rings.
If $\theta \colon S \rightarrow T$ is a homomorphism between semigroups
then $\mbox{\rm ker}(\theta)$ is the congruence that $\theta$ induces on $S$.
On the other hand, every congruence $\rho$ on a semigroup $S$ has an associated {\em natural homomorphism}
from $S$ to $S/\rho$ which maps $s$ to $\rho (s)$.
If both $S$ and $T$ are monoids then a {\em monoid homomorphism} maps the identity of $S$ to the identity of $T$;
if $S$ and $T$ are both semigroups with zero, then a homomorphism of such semigroups maps the zero to the zero.
If the semigroups $S$ and $T$ are isomorphic then we write $S \cong T$.
A subset $I \subseteq S$ is said to be a {\em right ideal} if $IS \subseteq I$
and a {\em left ideal} if $SI \subseteq I$.
It is said to be an {\em ideal} if $SI \subseteq I$ and $IS \subseteq I$.
If $a \in S$ then we write $a$ instead of $\{a\}$.
The right ideal $aS^{1}$ is called the {\em principal right ideal generated by $a$},
whereas the left ideal $S^{1}a$ is called the {\em principal left ideal generated by $a$}.
We call $S^{1}aS^{1}$ the {\em principal ideal generated by $a$}.
Our use of $S^{1}$ is just a trick to ensure that $a$ belongs to the ideal.
Define $a \,\mathscr{R}\, b$ if $aS^{1} = bS^{1}$ and
$a \,\mathscr{L}\, b$ if $S^{1}a = S^{1}b$.
Recall that $\mathscr{L} \circ \mathscr{R} = \mathscr{R} \circ \mathscr{L}$.
Put $\mathscr{D} = \mathscr{L} \circ \mathscr{R}$, another equivalence relation.
Put $\mathscr{H} = \mathscr{L} \cap \mathscr{R}$.
Define  $a \,\mathscr{J}\, b$ if $S^{1}aS^{1} = S^{1}bS^{1}$.
These are the familiar {\em Green's relations}.
As usual, if $\mathscr{G}$ is one of Green's relations, then the $\mathscr{G}$-class that contains the element $a$ is denoted by $G_{a}$.
Although ideals are useful in semigroup theory, the connection between ideals and congruences is weaker for semigroups than it is for rings.
If $\rho$ is a congruence on a semigroup with zero $S$, 
then the set $I = \rho (0)$ is an ideal of $S$;  however, examples show that the congruence is not determined by this ideal.
Nevertheless, ideals can be used to construct some congruences on semigroups.
Let $I$ be an ideal in the semigroup $S$.
Define a relation $\rho_{I}$ on $S$ by:
$$(s,t) \in \rho_{I} \Leftrightarrow \mbox{either } s,t \in I 
\mbox{ or } s = t.$$
Then $\rho_{I}$ is a congruence.
The quotient semigroup $S/\rho_{I}$ is isomorphic to the set
$(S \setminus I) \cup \{ 0\}$ 
(we may assume that $0 \notin S \setminus I$) 
equipped with the following product: 
if $s,t \in S\setminus I$ then their product is $st$ if 
$st \in S \setminus I$, all other products are defined to be $0$.
Such quotients are called {\em Rees quotients}.

There are currently two books entirely devoted to inverse semigroup theory:
Petrich's \cite{Petrich} and mine \cite{Lawson}.
Petrich's book is pretty comprehensive up to 1984 and is still a useful reference.
Its only drawback is the poor index which makes finding particular topics a bit of a chore.
My book is less ambitious;
its goal is to motivate the study of inverse semigroups by concentrating on concrete examples and was completed in 1998.

In writing this chapter, I have drawn mainly upon my own book \cite{Lawson}  
and some notes I wrote \cite{Lawson2020} for a workshop at the University of Ottawa in 2010,
but I have taken the opportunity to radically rethink what I wrote there
in the light of subsequent research.

Any results below that are stated but not proved should be treated as exercises.\\

\noindent
{\bf Acknowledgements.} I would like to thank Victoria Gould, Bernard Rybolowicz and Aidan Sims
for reading the first version of this chapter and giving me good notes.\\

\section{Motivation: pseudogroups of transformations}\setcounter{theorem}{0}

Good mathematics has to be based on good examples.
For example, finite semigroup theory grew out of the theory of finite-state automata
since a finite semigroup can be associated with every finite-state automaton.
Similarly, inverse semigroup theory grew out of the theory of pseudogroups of transformations,
which goes back to the works of Sophus Lie and Henri Cartan.
If you have ever taken an introductory course in differential geometry,
then you have been exposed to pseudogroups of transformations even if that phrase was
never once uttered. So, there I shall begin.
However, for those of you who have not taken such a course,
I shall quickly move on to the collection of all partial bijections on a set,
jettisoning any topology or differential structure, and this will take us to the core motivation for
studying inverse semigroups.

Let $X$ be a nice space (for example, $X$ should be Hausdorff but this will play no role in what follows).
Let $\mathbb{R}^{n}$ be the usual space of real $n$-vectors.
We shall coordinatize $X$ using $\mathbb{R}^{n}$ but we shall only do so {\em locally} not {\em globally}.
Thus we assume that for each point $x \in X$ there exists an open set $x \in U$ and a homeomorphism
$\phi \colon U \rightarrow V$, called a {\em chart}, where $V$ is an open subset of $\mathbb{R}^{n}$.
A chart gives one way of coordinatizing elements of $U$ by means of elements of $V$.
However,
it is quite possible that $x$ lies in another chart.
In other words, there could well exist another open subset of $X$, let's say $U'$,
which also contains $x$ and for which there is a homeomorphism $\psi \colon U' \rightarrow V'$ where
$V'$ is also an open subset of  $\mathbb{R}^{n}$.
This map gives another way of coordinatizing elements of $U'$ by means of elements of $V'$.
However, for any element in $y \in U \cap U'$, there are now two possible co-ordinate representations:
$\phi (y)$ and $\psi (y)$.
How are they to be related?
This depends on the nature of the map $\psi \phi^{-1} \colon \phi (U \cap U') \rightarrow \psi (U \cap U')$,
called the {\em change of co-ordinates map},
which connects them.
This map is a homeomorphism from the open subset $\phi (U \cap U')$ of $\mathbb{R}^{n}$ to the open subset $\psi (U \cap U')$ of $\mathbb{R}^{n}$.
If the map $\psi \phi^{-1}$ is always assumed to be {\em smooth} then we say that we have defined
on $X$ the structure of a {\em differential manifold} \cite{Sontz}.

We now dispense with $X$ and concentrate instead on the set $\Gamma (\mathbb{R}^{n})$ of all smooth maps between the open subsets of $\mathbb{R}^{n}$.
This is just the set of all possible changes of co-ordinates and is our first example of a pseudogroup.
Usually, it is the set $X$ which is the centre of attention,
but we have shifted focus to the set of allowable co-ordinate transformations;
if we change the nature of those, then we change the nature of the structure we are defining on $X$.

More generally, let $X$ be any topological space with topology $\tau$.
Denote the set of all homeomorphisms between open subsets of $X$ by $\mathcal{I}(X,\tau)$.
If $\tau$ is the usual topology on $\mathbb{R}^{n}$ then $\Gamma (\mathbb{R}^{n})$ is a subset of $\mathcal{I}(\mathbb{R}^{n},\tau)$
closed under certain operations.
The sets  $\mathcal{I}(X,\tau)$ are also examples of pseudogroups.
The terminology `pseudogroup' comes from the fact that, just like groups of transformations,
the sets $\mathcal{I}(X,\tau)$ are closed under the composition of partial functions 
--- 
don't forget the empty function! 
--- 
and are also closed under `inverses' where we interpret inverses as inverses of partial functions.
The set  $\Gamma (\mathbb{R}^{n})$ is likewise closed under composition of partial functions
and said inverses.
There is a further property that is needed to qualify as a pseudogroup.
It is this.
If $\{\phi_{i} \colon i \in I\}$ is any subset of $\mathcal{I}(X,\tau)$
such that $\bigcup_{i \in I} \phi_{i}$ is a partial bijection then, in fact, $\bigcup_{i \in I} \phi_{i} \in \mathcal{I}(X,\tau)$.
This is a completeness property.
We shall call semigroups of the form  $\mathcal{I}(X,\tau)$ {\em full pseudogroups of transformations}.

For modern applications of pseudogroups, see \cite{YC, MB}.

To make our lives easier, let's now take the non-empty set $X$ with the discrete topology.
We denote by $\mathcal{I}(X)$ the set of all bijections between subsets of $X$.
If $f \colon A \rightarrow B$, where $A,B \subseteq X$,
then we call $A$ the {\em domain (of definition)} of $f$
and $B$ the {\em range} of $f$.
We write $\mbox{\rm dom}(f) = A$ and $\mbox{\rm ran}(f) = B$.
The set $\mathcal{I}(X)$ is closed under the composition of partial functions
and so is a semigroup.
It is, in fact, a monoid since the identity function defined on $X$ belongs to  $\mathcal{I}(X)$.
By analogy with the symmetric group, it is called the {\em symmetric inverse monoid}.
What we mean by the term `inverse monoid' will be explained in the next section.
Clearly, the structure of the semigroup  $\mathcal{I}(X)$ depends only on the cardinality of $X$.
Thus, if $X$ is a finite set with $n$ elements, it makes sense to write $\mathcal{I}_{n}$ instead of $\mathcal{I}(X)$.

There are a number of structures we could look at in the semigroups  $\mathcal{I}(X)$ 
but we begin with the following which turns out to be the right place to start.
Let $f \in \mathcal{I}(X)$. 
Consider the following two equations in $\mathcal{I}(X)$:
$$f = fgf \text{ and } g = gfg.$$  
These have the unique solution $g = f^{-1}$.
This result is the basis for the definition of an inverse semigroup in the next section.

\section{Basic inverse semigroup theory}

A semigroup $S$ is said to be {\em inverse} if for each $s \in S$ there exists a unique element,
denoted by $s^{-1}$, such that the following two equations are satisfied: 
$$s = ss^{-1}s \text{ and } s^{-1} = s^{-1}ss^{-1}.$$

\begin{example}{\em The semigroups $\mathcal{I}(X)$ really are inverse monoids.
If $f \in \mathcal{I}(X)$ then $f^{-1}$ is the unique solution of the equations
$f = fgf$ and $g = gfg$ where $g$ is the unknown.}
\end{example}

Observe that if $s^{-1}$ is the inverse of $s$, then $s$ is the inverse of $s^{-1}$.
We have therefore proved the following.

\begin{lemma}\label{lem:inverse-of-inverse} Let $S$ be an inverse semigroup and let $s \in S$.
Then $(s^{-1})^{-1} = s$.
\end{lemma}  

The elements $s^{-1}s$ and $ss^{-1}$ are idempotents,
where an {\em idempotent} in a semigroup is an element $e$ such that $e^{2} = e$.
For example, $s^{-1}s$ is an idempotent because
$$(s^{-1}s)^{2} = (s^{-1}s)(s^{-1}s) = (s^{-1}ss^{-1})s = s^{-1}s.$$
If $e$ is an idempotent in an inverse semigroup $S$, then $e = ee = eee$.
We have therefore proved the following.

\begin{lemma}\label{lem:inverse-of-idempotent} Let $S$ be an inverse semigroup and let $e \in S$
be an idempotent.
Then $e^{-1} = e$.
\end{lemma}  

It follows that every idempotent in an inverse semigroup is of the form $s^{-1}s$ for some elements $s$ or,
equivalently, of the form $ss^{-1}$, for some element $s$.

Idempotents play an important role in inverse semigroup theory.
For that reason we introduce some special notation. 
If $S$ is an inverse semigroup we denote its set of idempotents by $\mathsf{E}(S)$.
Two special idempotents are the {\em identity} element, if it exists, usually denoted by $1$, and the {\em zero} element, if it exists,
usually denoted by $0$. 
An inverse semigroup with identity is called an {\em inverse monoid}
and an inverse semigroup with zero is called an {\em inverse semigroup with zero}.\footnote{There is no one term in English for a `semigroup with zero'.}
It is important to be aware right at the beginning that even if $S$ is an inverse monoid,
it is not true in general that $s^{-1}s = 1 = ss^{-1}$.

\begin{example}{\em The idempotents in $\mathcal{I}(X)$ are precisely the identity functions
defined on the subsets of $X$. Thus if $A \subseteq X$ then a typical idempotent is $1_{A}$
which is the identity function on the set $A$.
The identity is $1_{X}$ and the zero is $1_{\varnothing}$.}
\end{example}

The difference between monoids and semigroups is sometimes viewed as trivial.
It isn't.
For example, a commutative unital $C^{\ast}$-algebra is constructed from a {\em compact} space whereas
a commutative non-unital $C^{\ast}$-algebra is constructed from a {\em locally compact} space.

If $S$ is any monoid then we can look at its {\em group of units} which we denote by $\mathsf{U}(S)$.
If $e$ is any idempotent in a semigroup $S$ then $eSe$ is a monoid.
I call these {\em local monoids} though you will find the misleading term `local submonoid' in the literature.
The group of units of $eSe$ is denoted by $H_{e}$.

\begin{example} {\em The units of $\mathcal{I}(X)$ are the bijective functions and these form a group $\mathcal{S}(X)$, called the
{\em symmetric group} on $X$.}
\end{example}

\begin{lemma}\label{lem:local} Let $S$ be an inverse semigroup and let $e$ be an idempotent in $S$.
Then $eSe$ is an inverse monoid.
\end{lemma} 
\begin{proof} It is clear that $eSe$ is a monoid.
We prove that it is inverse.
Let $a \in eSe$.
Then $a = aa^{-1}a$ and $a^{-1} = a^{-1}aa^{-1}$ in $S$.
You can check that $ea^{-1}e$ is also an inverse of $a$ since $ea = a = ae$.
So, by uniqueness, we have that $a^{-1} = ea^{-1}e$.
We have therefore proved that if $a \in eSe$ then $a^{-1} \in eSe$.
It follows that $eSe$ is an inverse monoid.
\end{proof}

The idempotents $s^{-1}s$ and $ss^{-1}$ are so important, that we have some special notation for them:
$$\mathbf{d}(s) = s^{-1}s \text{ and } \mathbf{r}(s) = ss^{-1}.$$
Observe that $a \mathbf{d}(a) = a$ and $\mathbf{r}(a)a = a$.
If the inverse semigroup has a zero then $a = 0$ if and only if $\mathbf{d}(a) = 0$ (respectively, $\mathbf{r}(a) = 0$.) 
You can easily check that in an inverse semigroup
$a \,\mathscr{R}\, b$ if and only if $\mathbf{r}(a) = \mathbf{r}(b)$ and
$a \,\mathscr{L}\, b$ if and only if $\mathbf{d}(a) = \mathbf{d}(b)$.
For this reason, the explicit use of Green's relations in inverse semigroups
is not very common.

An {\em inverse subsemigroup} of an inverse semigroup is a subsemigroup that is also closed under inverses.
If $S$ is an inverse subsemigroup of $T$ and $\mathsf{E}(S) = \mathsf{E}(T)$ we say that $S$ is a {\em wide}\footnote{You will often see the term `full' used
to mean `wide'. I prefer the term `wide' because, as we shall see, it fits in well with groupoid theory, whereas `full' has categorical connotations
which are not what we want.} inverse subsemigroup of $T$.

\begin{example}{\em The inverse semigroup $\Gamma (\mathbb{R}^{n})$ is, in fact,
a wide inverse submonoid of $\mathcal {I}(\mathbb{R}^{n},\tau)$ where $\tau$ is the usual
topology on $\mathbb{R}^{n}$.
}
\end{example}

It is now time for our first example.
We can, in fact, express this as a lemma.

\begin{lemma}\label{lem:unique-idempotents} \mbox{}
\begin{enumerate}
\item A group is an inverse semigroup having a unique idempotent.
\item An inverse semigroup having a unique idempotent is a  group.
\end{enumerate}
\end{lemma}
\begin{proof}
(1) Let $G$ be a group. Then for any element $g$ in $G$ we have that $g^{-1}g = 1 = gg^{-1}$.
Multiply the first equation on the left by $g$ to get $gg^{-1}g = g$
and multiply the second equation on the left by $g^{-1}$ to get $g^{-1}gg^{-1} = g^{-1}$.
On the other hand, suppose that $gxg = g$ and $xgx = x$.
Multiply the first equation by $g^{-1}$ on the left and by $g^{-1}$ on the right.
This gives us $x = g^{-1}$ which also satisfies the second equation.
It follows that every group is an inverse semigroup.
Let $g$ be an idempotent in $G$.
Then $gg = g$.
Multiply the left-hand side of this equation by $g^{-1}$ to get $g = 1$,
and $1$ is clearly an idempotent.
Thus the only idempotent in a group is the identity. 

(2) Let $S$ be an inverse semigroup with exactly one idempotent.
Call this idempotent $e$.
For any element $a \in S$, we already know that $a^{-1}a$ and $aa^{-1}$ are idempotents.
It follows that $a^{-1}a = e = aa^{-1}$ for any element $a \in S$.
On the other hand, $e$ is the identity.
To see why, observe that $ea = (aa^{-1})a = a$.
Similarly, $ae = a$.
We have therefore proved that $S$ is a group with identity $e$. 
\end{proof}

The most important fact about inverse semigroups,
and one that is not at all obvious from the definition,
is that if $e$ and $f$ are any idempotents in an inverse semigroup $S$ then $ef = fe$.
We say that the {\em idempotents commute}.
For the benefit of ring theorists, observe that we are not saying that the idempotents are central.
Inverse semigroups in which the idempotents are central are called {\em Clifford semigroups}.

\begin{proposition}\label{prop:idempotents-commute} 
Idempotents commute in an inverse semigroup.
\end{proposition}
\begin{proof} This is an example of a proof that is elementary but not easy.
Let $e$ and $f$ be idempotents in the inverse semigroup $S$.
Put $x = (ef)^{-1}$.
We prove first that $fxe$ is an idempotent:
$$(fxe)^{2} = (fxe)(fxe) = f(x(ef)x)e = fxe.$$  
Next, we claim that $fxe$ is the inverse of $ef$:
$$ef(fxe)ef = ef^{2}xe^{2}f = efxef = ef$$
and
$$(fxe)ef(fxe) = fxe^{2}f^{2}xe = (fxe)^{2} = fxe,$$ 
by what we proved above.
It follows that 
$$(ef)^{-1} = fxe.$$
We have therefore proved that $(ef)^{-1}$ is an idempotent.
Now, we apply  
Lemma~\ref{lem:inverse-of-inverse} 
and 
Lemma~\ref{lem:inverse-of-idempotent} 
to deduce that $ef$ is an idempotent.
We have therefore proved that in an inverse semigroup
the product of any two idempotents is itself an idempotent.
It follows that both $ef$ and $fe$ are idempotents.
Next, we show that $fe$ is the inverse of $ef$:
$$(ef)(fe)(ef) = ef^{2}e^{2}f = efef = ef$$
and
$$(fe)(ef)(fe) = fe^{2}f^{2}e = fefe = fe.$$
We have therefore proved that $(ef)^{-1} = fe$.
But $ef$ is an idempotent and so, by Lemma~\ref{lem:inverse-of-idempotent}, it is its own inverse.
We have therefore proved that $ef = fe$, as required. 
\end{proof}

\begin{example}{\em In the symmetric inverse monoid, the product of the idempotents
$1_{A}$ and $1_{B}$ is the idempotent $1_{A \cap B}$.
Because the intersection operation on two sets is commutative, we see why the idempotents commute in this special case.}
\end{example}

Because we know that idempotents commute, we can now prove the following:

\begin{lemma}\label{lem:inverse-of-product} 
Let $S$ be an inverse semigroup.
Then $(st)^{-1} = t^{-1}s^{-1}$.
\end{lemma}  
\begin{proof} We prove that $t^{-1}s^{-1}$ is the inverse of $st$.
We have that
$$st(t^{-1}s^{-1})st = s(tt^{-1})(s^{-1}s)t = s(s^{-1}s)(tt^{-1})t = st$$
where we have used the fact that idempotents in an inverse semigroup commute.
Similarly,
$$t^{-1}s^{-1}(st)t^{-1}s^{-1} = t^{-1}s^{-1}.$$
By virtue of the fact that inverses are unique, we deduce that $(st)^{-1} = t^{-1}s^{-1}$.
\end{proof}

We also see that the analogue of conjugation holds.
We shall use the term `conjugation' below.

\begin{lemma}\label{lem:conjugation-of-idpts} 
If $e$ is an idempotent of an inverse semigroup and $s$ is any element of that inverse semigroup then $ses^{-1}$ is an idempotent.
\end{lemma}
\begin{proof} This follows by a direct calculation
$$(ses^{-1})^{2} = ses^{-1}ses^{-1} = se^{2}(s^{-1}s)s^{-1} = ses^{-1},$$
where we have again used the fact that idempotents commute.
\end{proof}

Let $e$ and $f$ be idempotents in an inverse semigroup.
Then we may define a relation on them by $e \leq f$ precisely when $e = ef$.
Because idempotents commute, it is not hard to show that this defines a partial order 
on the set of idempotents. We can now characterize Clifford semigroups.

\begin{lemma}\label{lem:characterization-clifford}
Let $S$ be an inverse semigroup.
Then $S$ is a Clifford semigroup if and only if 
$\mathbf{d}(s) = \mathbf{r}(s)$
for every $s \in S$.
\end{lemma}
\begin{proof} We use the partial order we have defined on the idempotents of an inverse semigroup.
Let $S$ be a Clifford semigroup and let $s \in S$.
Since the idempotents are central
$s = s(s^{-1}s) = (s^{-1}s)s$.
Multiply the last equality on the right by $s^{-1}$
to deduce that $\mathbf{r}(s) \leq \mathbf{d}(s)$.
It now follows by symmetry that $\mathbf{d}(s) = \mathbf{r}(s)$.

Suppose now that $\mathbf{d}(s) = \mathbf{r}(s)$ for all $s \in S$.
Let $e$ be any idempotent and $s$ arbitrary.
We shall prove that $se = es$.
By assumption, $\mathbf{d}(es) = \mathbf{r}(es)$.
That is $(es)^{-1}es = es(es)^{-1}$.
We now use Lemma~\ref{lem:inverse-of-product} and Lemma~\ref{lem:inverse-of-idempotent}
and the fact that idempotents commute to get
$s^{-1}es = ss^{-1}e$.
Because of our assumption, $ss^{-1} = s^{-1}s$.
Thus 
$s^{-1}es = s^{-1}se$.
Multiplying on the left by $s$, 
and using the fact that idempotents commute and our assumption, gives $es = se$, as required.
\end{proof}

An inverse semigroup $S$ is said to be a {\em union of groups} if
$S = \bigcup_{e \in \mathsf{E}(S)} H_{e}$. 
By Lemma~\ref{lem:characterization-clifford}, we may now deduce the following

\begin{lemma}\label{lem:union-of-groups} 
Let $S$ be an inverse semigroup.
Then $S$ is a Clifford semigroup if and only if it is a union of groups.
\end{lemma}
\begin{proof}
Let $S$ be a Clifford semigroup and let $s \in S$.
Put $e = \mathbf{d}(s) = \mathbf{r}(s)$.
Then $s \in eSe$.
But, in fact, $s$ belongs to the group of units of $eSe$ which is $H_{e}$.
Thus $s \in H_{e}$ and so we have proved that every Clifford semigroup
is a union of groups.

Conversely, suppose that $S$ is an inverse semigroup which is a union of groups.
Then for each $s \in S$, we have that $s \in H_{e}$ for some idempotent $e$.
We have that $\mathbf{d}(s) = e = \mathbf{r}(s)$.
It follows that $S$ Clifford.
\end{proof}

I have included Lemma~\ref{lem:union-of-groups}  
because group theorists sometimes tend to see
inverse semigroups as unions of groups.
These are, in fact, a very special class of inverse semigroups.
Clifford semigroups can be described explicitly as `strong semilattices of groups'
or, what amounts to the same thing, as presheaves of groups over meet semilattices.
This applies to, in particular, abelian inverse semigroups.
For strong semilattices of groups, see \cite[Chapter IV, Section 2]{Howie}.

We can obtain another characterization of Clifford semigroups
which is analogous to part of \cite[part (c) of Theorem 3.2]{Goodearl} but at the price of an extra assumption on the semilattice of idempotents.
We say that a meet semilattice is {\em $0$-disjunctive} if 
for all $0 \neq f < e$ there exists $0 \neq g \leq e$ such that $fg = 0$.
Thus, semilattices which are $0$-disjunctive have a weak notion of complement.
A non-zero element $a$ of an inverse semigroup is said to be an {\em infinitesimal} if $a^{2} = 0$.

\begin{lemma}\label{lem:clifford-infinitesimal} 
Let $S$ be an inverse semigroup, the semilattice of idempotents of which is $0$-disjunctive.
Then $S$ is a Clifford semigroup if and only if $S$ contains no infinitesimals.
\end{lemma}
\begin{proof} Suppose that $S$ is a Clifford semigroup.
Let $a$ be an infinitesimal.
Then $a^{2} = 0$ thus $aa = 0$.
It follows that $\mathbf{d}(a)\mathbf{r}(a) = 0$.
But, in a Clifford semigroup $\mathbf{d}(a) = \mathbf{r}(a)$ by Lemma~\ref{lem:characterization-clifford}.
It follows that $\mathbf{d}(a) = 0$ and so $a = 0$ which is a contradiction.
Thus there are no infinitesimals.

To prove the converse, suppose that there are no infinitesmials.
Let $a \in S$ be arbitrary and non-zero.
We shall prove that $\mathbf{d}(a) = \mathbf{r}(a)$.
Suppose not.
Put $e = \mathbf{d}(a)\mathbf{r}(a)$.
If $e = 0$ then $a^{-1}aaa^{-1} = 0$ and so $a^{2} = 0$.
But this means that $a$ is an infinitesimal.
It follows that $e \neq 0$.
There are now two cases.
Suppose, first, that $e = \mathbf{d}(a)$.
Then $\mathbf{d}(a) < \mathbf{r}(a)$. 
By assumption, there exists $0 < f \leq \mathbf{r}(a)$ such that
$f\mathbf{d}(a) = 0$.
Put $b = fa$.
If $b = 0$ then it is easy to see that $f = 0$, which is a contradiction.
It follows that $b \neq 0$.
You can easily check that $b$ is an infinitesimal.
This is a contradiction.
We can now deal with the second case where $e < \mathbf{d}(a)$.
By assumption, there is a non-zero idempotent $f$ such that
$ef = 0$ and $f \leq \mathbf{d}(a)$.
Put $b = af$.
If $b = 0$ then $(af)^{-1}af = 0$ which implies that $f = 0$.
It follows that $b \neq 0$.
However, $b^{2} = afaf = a \mathbf{d}(a)f \mathbf{r}(a)af = aefaf = 0$.
But this contradicts the assumption that there are no infinitesimals.
It follows that $\mathbf{d}(a) = \mathbf{r}(a)$ for all $a \in S$.
Thus by Lemma~\ref{lem:characterization-clifford}, we have shown that $S$ is a Clifford semigroup.
\end{proof}

Although the idempotents in an inverse semigroup are not in general central, 
we do have the following result, which tells us how idempotents `pass through' non-idempotent elements.

\begin{lemma}\label{lem:pass-through} Let $S$ be an inverse semigroup in which $e$ is an idempotent and $s$ is any element.
\begin{enumerate}
\item $es = sf$ for some idempotent $f$.
\item $se = gs$ for some idempotent $g$.
\end{enumerate}
\end{lemma}
\begin{proof} We prove only (1) since the proof of (2) is similar.
We have that $es = ess^{-1}s = e(ss^{-1})s$.
Now we use the fact that idempotents commute to get $es = (ss^{-1})es$.
But this is equal to $s(s^{-1}es)$ and now we use Lemma~\ref{lem:conjugation-of-idpts} 
to deduce that $f = s^{-1}es$ is an idempotent.
We have therefore proved that $es = sf$.
\end{proof}

We have described one extreme example of inverse semigroup.
We now describe another.
By a {\em meet semilattice}, we mean a partially ordered set $(P,\leq)$ with the property that
every pair of elements $a,b \in P$ has a {\em greatest lower bound}, denoted by $a \wedge b$.
Thus, $a \wedge b \leq a,b$ and if $c \leq a,b$ then $c \leq a \wedge b$.
We call $a \wedge b$ the {\em meet} of $a$ and $b$.\footnote{We shall usually denote the meet by $\cdot$ or concatenation.} 
The proofs of the following are immediate from the definitions.

\begin{lemma}\label{lem:meet-semilattice}
Let $(P,\leq)$ be a meet semilattice and let $a,b,c \in P$.
Then 
\begin{enumerate}
\item $a \wedge (b \wedge c) = (a \wedge b) \wedge c$.
\item $a \wedge b = b \wedge a$.
\item $a \wedge a = a$
\end{enumerate}
\end{lemma}

A semigroup is called a {\em band} if every element is an idempotent.
Observe that we have proved in Lemma~\ref{lem:meet-semilattice} that if $(P,\leq)$ is a meet semilattice then $(P,\wedge)$ is a commutative band.

\begin{lemma}\label{lem:commutative-band} Let $S$ be a commutative band.
Define $a \leq b$ by $a = ab$.
Then $(S,\leq)$ is a partially ordered set and, in fact, a meet semilattice in which $a \wedge b = ab$.
\end{lemma}
\begin{proof} We first check that $\leq$ is a partial order.
By assumption, each element of $S$ is an idempotent.
It follows that $a \leq a$.
Suppose that $a \leq b$ and $b \leq a$.
Then $a = ab$ and $b = ba$.
But the semigroup is commutative and so $a = b$.
Finally, suppose that $a \leq b$ and $b \leq c$.
Then $a = ab$ and $b = bc$.
It follows that $a = abc = ac$.
Thus $a \leq c$.
We have verified that $(S,\leq)$ is a partially ordered set.
We now show that $a \wedge b = ab$.
Observe first that $(ab)a = a^{2}b = ab$, using commutativity and the fact that every element is idempotent.
Also, $(ab)b = ab^{2} = ab$.
We have show that $ab \leq a,b$.
Suppose, now that $c \leq a,b$.
Then $c = ca = cb$, and so $c(ab) = cb = c$.
This shows that $c \leq ab$ and thus proves that $ab = a \wedge b$.
\end{proof}

As a result of Lemma~\ref{lem:meet-semilattice} and Lemma~\ref{lem:commutative-band},
we can think about meet semilattices in two equivalent ways:
either as partially ordered sets in which each pair of elements has a greatest lower bound
or as commutative bands.

Observe that if $S$ is any inverse semigroup then, since the idempotents commute, $\mathsf{E}(S)$ is a subsemigroup of $S$. 
It follows that $\mathsf{E}(S)$ is always a commutative band and so it becomes a meet semilattice when we define $e \wedge f = ef$.
For this reason, it is usual to refer to $\mathsf{E}(S)$ as the {\em semilattice of idempotents} of $S$.

We now have the following companion to Lemma~\ref{lem:unique-idempotents}. 

\begin{lemma}\label{lem:all-idempotents} \mbox{}
\begin{enumerate}
\item A meet semilattice is an inverse semigroup with respect to the meet operation in which every element is an idempotent.
\item An inverse semigroup in which every element is an idempotent is a meet semilattice.  
\end{enumerate}
\end{lemma}
\begin{proof} The proof of (1) is immediate by Lemma~\ref{lem:meet-semilattice}.
(2) Suppose that $S$ is an inverse semigroup in which every element is an idempotent.
Then $S = \mathsf{E}(S)$.
The result now follows by Lemma~\ref{lem:commutative-band},
since $\mathsf{E}(S)$ is a commutative band.
\end{proof}

Our next result is often useful in showing that a semigroup is or is not inverse.
We need a definition first.
A semigroup $S$ is said to be {\em regular} if for each $a \in S$ there exists an element $b$ such that $a = aba$ and $b = bab$.
The element $b$ is called {\em an inverse} of $a$.
Observe that both $ab$ and $ba$ are idempotents.
In showing that a semigroup is regular, it is enough to check that for each element $a$ there is an element $x$
such that $a = axa$.
The reason is that $b = xax$ is then an inverse of $a$.
Inverse semigroups are the regular semigroups in which every element has a unique inverse. 

\begin{proposition}\label{prop:regular-commuting-idempotents} 
A regular semigroup is inverse if and only if its idempotents commute.
\end{proposition}
\begin{proof}
Let $S$ be a regular semigroup in which the idempotents commute and  
let $u$ and $v$ be inverses of $x$.  
Then
$$u = uxu = u(xvx)u = (ux)(vx)u,$$
where both $ux$ and $vx$ are idempotents.  
Thus, since idempotents commute, we have that
$$u = (vx)(ux)u = vxu = (vxv)xu = v(xv)(xu).$$
Again, $xv$ and $xu$ are idempotents and so
$$u = v(xu)(xv) = v(xux)v = vxv = v. $$ 
Hence $u = v$.
The converse follows by Proposition~\ref{prop:idempotents-commute}. 
\end{proof}

\begin{example}{\em As an example of Proposition~\ref{prop:regular-commuting-idempotents} 
in action, observe that the full transformation monoid $\mathcal{T}(X)$, of all functions from $X$ to itself where $X$ has at least two elements, is regular but not inverse.
We show first that it is regular.
Let $f \in \mathcal{I}(X)$.
There are two cases.
Suppose first that $f$ is surjective;
for each $y \in X$ choose $x_{y} \in f^{-1}(y)$.
Define $g \colon X \rightarrow X$ by $g(y) = x_{y}$.
That is, $g$ maps an element to one of its pre-images under $f$.
We calculate $fgf$.
Let $x \in X$ and $y = f(x)$.
Then 
$$(fgf)(x) = (fg)(f(x)) = f(g(y)) = f(x_{y}) = y.$$  
It follows that $f = fgf$.
Suppose now that $f$ is not surjective.
Choose any element $x_{0} \in X$.
For each element $y$ in the image of $f$ choose  $x_{y} \in f^{-1}(y)$.
Now define $g \colon X \rightarrow X$ as follows:
if $y$ is in the image of $f$, then define $g(y) = x_{y}$,
and if $y$ is not in the image of $f$, then define $g(y) = x_{0}$.
A similar calculation to the one above shows that $f = fgf$.
We have therefore shown that the full transformation monoid is regular.
Choose two distinct elements from $X$ that we call $x_{1}$ and $x_{2}$.
Define two functions $c_{1},c_{2} \colon X \rightarrow X$
where $c_{1}$ is the constant function to $x_{1}$
and $c_{2}$ is the constant function to $x_{2}$.
The functions $c_{1}$ and $c_{2}$ are distinct and both are idempotents.
But $c_{1}c_{2} = c_{1}$ whereas $c_{2}c_{1} = c_{2}$.
We have shown that the idempotents do not commute.
Thus $\mathcal{T}(X)$ is regular but not inverse.}
\end{example}

We now consider homomorphisms between inverse semigroups.
Our first result is actually slightly more general.

\begin{lemma}\label{lem:homomorphisms} Let $S$ be an inverse semigroup and let $T$ be any semigroup.
Suppose that $\theta \colon S \rightarrow T$ is a semigroup homomorphism.
Then $\theta (S)$, the image of $S$, is an inverse semigroup.
\end{lemma}
\begin{proof} We shall use Proposition~\ref{prop:regular-commuting-idempotents}  to prove that $\theta (S)$ is inverse.
We show first that it is regular.
Let $\theta (s) \in \theta (S)$.
Then, since $s \in S$ inverse, there is an element $s^{-1} \in S$ such that
$s = ss^{-1}s$ and $s^{-1} = s^{-1}ss^{-1}$.
It follows that 
$\theta (s) = \theta (s) \theta (s^{-1}) \theta (s)$
and 
$\theta (s^{-1}) = \theta (s^{-1}) \theta (s) \theta (s^{-1})$.
This shows that $\theta (S)$ is regular.

We now prove that the idempotents in $\theta (S)$ commute.
To begin with, we have to show that idempotents in the image come from idempotents 
in the domain.\footnote{This is an instance of what is often termed {\em Lallement's Lemma}.}
Let $\theta (s)$ be an idempotent in $\theta (S)$.
We shall prove that there is an idempotent $e$ in $S$ such that $\theta (e) = \theta (s)$.
Consider the element $e = ss^{-2}s = (ss^{-1})(s^{-1}s)$.
It is the product of two idempotents in $S$ and so it is an idempotent.
In particular, $e$ is an idempotent in $S$.
We calculate $\theta (e)$.
We have that
$$\theta (e) = \theta (ss^{-2}s) = \theta (s)\theta (s^{-2})\theta (s).$$ 
But $\theta (s) = \theta (s)^{2} = \theta (s^{2})$.
We therefore have that 
$$\theta (e) = \theta (s^{2}s^{-2}s^{2}) = \theta (s^{2}) = \theta (s)^{2} = \theta (s).$$
Now, let $\theta (s)$ and $\theta (t)$ be idempotents in $\theta (S)$.
By our result above, there are idempotents $e$ and $f$ in $S$ such that
$\theta (e) = \theta (s)$ and $\theta (f) = \theta (t)$.
It follows that $\theta (ef) = \theta (s)\theta (t)$.
But idempotents in inverse semigroups commute and so $ef = fe$.
We have therefore shown that $\theta (s) \theta (t) = \theta (t) \theta (s)$.
Thus, the idempotents in $\theta (S)$ commute.
We have therefore proved that the image of $\theta$ is an inverse semigroup by Proposition~\ref{prop:regular-commuting-idempotents}. 
\end{proof}

{\em Homomorphisms of inverse semigroups} are just semigroup homomorphisms.
{\em Isomorphisms of inverse semigroups} are just semigroup isomorphisms.

\begin{lemma}\label{lem:hm-inverse} 
Let $\theta \colon S \rightarrow  T$ be a homomorphism  between inverse semigroups. 
Then: 
\begin{enumerate}

\item $\theta (s^{-1}) = \theta (s)^{-1}$ for all $s \in S$. 

\item If $e$ is an idempotent then $\theta (e)$ is an idempotent. 

\item If $\theta (s)$ is an idempotent then there is 
an idempotent $e$ in $S$
such that $\theta (s) = \theta (e)$. 

\item $\theta (S)$ is an inverse subsemigroup of $T$. 

\item If $U$ is an inverse subsemigroup of $T$
then $\theta^{-1}(U)$ is an inverse subsemigroup of $S$.

\end{enumerate}
\end{lemma}
\begin{proof}
(1) Clearly, 
$\theta (s) \theta (s^{-1}) \theta (s) = \theta (s)$ and 
$\theta (s^{-1}) \theta (s) \theta (s^{-1}) = \theta (s^{-1})$.   
Thus, by uniqueness of inverses, we have that 
$\theta (s^{-1}) = \theta (s)^{-1}$.

(2) $\theta (e)^{2} = \theta (e) \theta (e) = \theta (e^{2}) = \theta (e)$. 

(3) If $\theta (s)^{2} = \theta (s)$, then 
$\theta (s^{-1}s) = \theta (s^{-1}) \theta (s) 
= \theta (s)^{-1} \theta (s) = \theta (s)^{2} = \theta (s)$,
where we have used the fact that every idempotent in an inverse semigroup is its own inverse.

(4) This is immediate by Lemma~\ref{lem:homomorphisms}.

(5) Straightforward.
\end{proof}

Let $\theta \colon S \rightarrow T$ be a homomorphism of inverse semigroups.
Then the restriction map, $(\theta \mid \mathsf{E}(S)) \colon \mathsf{E}(S) \rightarrow \mathsf{E}(T)$,
is well-defined since homomorphisms map idempotents to idempotents by part (2) of Lemma~\ref{lem:hm-inverse} 
The homomorphism $\theta$ is said to be {\em idempotent-separating}
if  $(\theta \mid \mathsf{E}(S))$ is actually injective.
We say that a congruence is idempotent-separating if its associated natural homomorphism 
is idempotent-separating.
We say that $S$ is an {\em idempotent-separating cover} of $T$
if $\theta$ is both surjective and idempotent-separating.

It remains to be shown that inverse semigroups really do encode partial bijections.
Given an inverse semigroup $S$ our goal is therefore to construct an injective homomorphism
$\lambda \colon S \rightarrow \mathcal{I}(X)$ for some set $X$.
We shall do just this by generalizing the famililar Cayley's Theorem from group theory.
As a starting point, we follow the group theory and take as our underlying set $X$ the set $S$ itself.
We now have to associate an element of $\mathcal{I}(S)$ with each element $a \in S$.
This is delivered by the following lemma.

\begin{lemma}\label{lem:WP-one} Let $S$ be an inverse semigroup and let $a \in S$.
Then $\lambda_{a} \colon \mathbf{d}(a)S \rightarrow  \mathbf{r}(a)S$ defined by
by $\lambda_{a}(x) = ax$ is a well-defined partial bijection.
\end{lemma}  
\begin{proof}
This is well-defined because $aS = aa^{-1}S$ as the following set inclusions show
$$aS = aa^{-1}aS \subseteq aa^{-1}S \subseteq aS.$$  
Also
$\lambda_{a^{-1}} \colon \mathbf{r}(a)S \rightarrow \mathbf{d}(a)S$,
$\lambda_{a^{-1}} \lambda_{a}$ is the identity on $\mathbf{d}(a)S$,
and $\lambda_{a} \lambda_{a^{-1}}$ is the identity on $\mathbf{r}(a)S$.
Thus $\lambda_{a}$ is a bijection and   
$\lambda_{a}^{-1} = \lambda_{a^{-1}}$.
\end{proof}

We now define a function $\lambda \colon S \rightarrow \mathcal{I}(S)$
by $\lambda (a) = \lambda_{a}$.
The following theorem completes the circle of ideas begun in Section~2
and justifies the claim that inverse semigroups are the abstract
theory of partial bijections.

\begin{theorem}[Wagner-Preston representation theorem]\label{them:W-P}
Every inverse semigroup can be embedded in a symmetric inverse monoid.
\end{theorem}
\begin{proof} We prove first that $\lambda$ is a semigroup homomorphism.
To that end, we prove that $\lambda_{a} \lambda_{b} = \lambda_{ab}$.
If $e$ and $f$ are any idempotents of an inverse semigroup $S$ then
$$eS \cap fS = efS.$$
Thus
$$\mbox{dom}(\lambda_{a}) \cap \mbox{im}(\lambda_{b}) 
= a^{-1}aS \cap bb^{-1}S = a^{-1}abb^{-1}S.$$
Hence 
$$\mbox{dom}(\lambda_{a}  \lambda_{b}) 
= \lambda_{b}^{-1}(a^{-1}abb^{-1}S) = b^{-1}a^{-1}aS
= b^{-1}a^{-1}abS$$
where we use the following subset inclusions
$$b^{-1}a^{-1}aS = b^{-1}bb^{-1}a^{-1}aS = b^{-1}a^{-1}abb^{-1}S \subseteq b^{-1}a^{-1}abS \subseteq b^{-1}a^{-1}aS.$$
Thus $\mbox{dom}(\lambda_{a} \lambda_{b}) = \mbox{dom}(\lambda_{ab})$.
It is immediate from the definitions that 
$\lambda_{a} \lambda_{b}$ and $\lambda_{ab}$ 
have the same effect on elements,
and so $\lambda$ is a homomorphism.
It remains to prove that $\lambda$ is injective.
Suppose that $\lambda_{a} = \lambda_{b}$.
Then $a = ba^{-1}a$ and $b = ab^{-1}b$.
Observe that $ab^{-1}b = (ba^{-1}a)b^{-1}b = b(b^{-1}b)a^{-1}a = ba^{-1}a = a$,
where we have used the fact that idempotents commute.
It follows that $a = b$.
\end{proof}

We conclude this section by stating a deeper result that we shall not prove.
Let $S$ be a {\em finite} semigroup.
Then for each $s \in S$ there exists some power $s^{n}$ which is an idempotent.
Let $\theta \colon S \rightarrow T$ be a homomorphism between finite semigroups.
Suppose that $\theta (s)$ is an idempotent in $T$.
Then for some natural number $n \geq 1$,
the element $s^{n}$ is also an idempotent.
But $\theta (s^{n}) = \theta (s)^{n} = \theta (s)$.
Thus, for finite semigroups, if $\theta (s)$ is an idempotent then there is an idempotent $e$ 
such that $\theta (e) = \theta (s)$.
If $S$ is an inverse semigroup then every subsemigroup (notice: I did not say {\em inverse} subsemigroup on purpose)
has commuting idempotents.
We say that a semigroup $T$ {\em divides} a semigroup $S$ if there is a subsemigroup
$S'$ of $S$ such that $T$ is a homomorphic image of $S'$.
Suppose now that $S$ is a finite inverse semigroup.
Then every subsemigroup $S'$ of $S$ is a finite semigroup that has commuting idempotents.
Thus every homomorphic image of $S'$ is a finite semigroup with commuting idempotents.
We have therefore proved the following:

\begin{lemma}\label{lem:finite-commuting-idempotents}
Every semigroup that divides a finite inverse semigroup has commuting idempotents.
\end{lemma}

The following theorem was proved by Chris Ash \cite{Ash} and is deep.
It is the converse of the above lemma.

\begin{theorem}[Ash]\label{them:ash-deep-theorem} Every finite semigroup with commuting idempotents
divides a finite inverse semigroup.
\end{theorem}

This is a beautiful result, and a stunning piece of mathematics.
It uses finite combinatorics, in the guise of Ramsey Theory, as part of the proof.
See \cite{Lallement} for the connections between combinatorics and semigroup theory.

\section{The natural partial order}

Partial bijections can be compared with each other: we can say that one partial bijection is the restriction of another.
This leads to a partial order on the set of partial bijections of a set.
It might be thought that this partial order has to be imposed on an inverse semigroup but,
remarkably, this partial order can be defined in purely algebraic terms.
It is therefore called the natural partial order.
It will follow that every inverse semigroup is, in fact, a partially ordered semigroup
with respect to this order.

On an inverse semigroup, define $s \leq t$ iff $s = ts^{-1}s$. 
This looks one-sided because the idempotent appears on the right-hand side.
It is here that we invoke Lemma~\ref{lem:pass-through}.

\begin{lemma}\label{lem:carol} In an inverse semigroup, the following are equivalent:
\begin{enumerate}

\item $s \leq t$.

\item $s = te$ for some idempotent $e$.

\item $s = ft$ for some idempotent $f$.

\item $s = ss^{-1}t$.

\end{enumerate}
\end{lemma}
\begin{proof} 

(1)$\Rightarrow$(2). This is immediate.

(2)$\Rightarrow$(3). This is immediate by Lemma~\ref{lem:pass-through}. 

(3)$\Rightarrow$(4). Suppose that $s = ft$.
Then $fs = s$ and so $fss^{-1} = ss^{-1}$.
It follows that $s = ss^{-1}t$.

(4)$\Rightarrow$(1). Suppose that $s = ss^{-1}t$.
By Lemma~\ref{lem:pass-through}, we know that $s = ti$ for some idempotent $i$. 
Observe that $si = s$ and so $s^{-1}si = s^{-1}s$.
It readily follows that $s = ts^{-1}s$ giving $s \leq t$.
\end{proof}

Let $S$ be a semigroup equipped with a partial order $\leq$.
We say that $S$ is a {\em partially ordered semigroup}
if $a \leq b$ and $c \leq d$ imply that $ac \leq bd$. 
We may now establish the main properties of the relation we have defined.

\begin{lemma}\label{lem:po} Let $S$ be an inverse semigroup.
\begin{enumerate}

\item The relation $\leq$ is a partial order.

\item If $s \leq t$ then $s^{-1} \leq t^{-1}$.

\item The semigroup $S$ is partially ordered with respect to $\leq$.

\item If $e$ and $f$ are idempotents then $e \leq f$ if and only if $e = ef = fe$.

\item $s \leq e$, where $e$ is an idempotent, implies that $s$ is an idempotent.

\item Let $T$ be an inverse semigroup and let $\theta \colon S \rightarrow T$ be a homomorphism.
Then $a \leq b$ in $S$ implies that $\theta (a) \leq \theta (b)$ in $T$.

\end{enumerate}
\end{lemma}
\begin{proof} (1) Observe that $a \leq a$ since $a = a\mathbf{d}(a)$.
Suppose that $a \leq b$ and $b \leq a$.
Then $a = b\mathbf{d}(a)$ and $b = a\mathbf{d}(b)$.
Then $b\mathbf{d}(a) = b$ from which it follows that $a = b$.
Suppose that $a \leq b$ and $b \leq c$.
Using Lemma~\ref{lem:carol}, it is easy to show that $a \leq c$.

(2) This is immediate using the definition and Lemma~\ref{lem:carol}.

(3) This is immediate using the definition, Lemma~\ref{lem:carol} and Lemma~\ref{lem:pass-through}.

(4) Straightforward.

(5) This is immediate from the definition and the fact that the product of idempotents is an idempotent.

(6) This is immediate from the fact that the natural partial order is algebraically defined, Lemma~\ref{lem:carol} and the fact
that idempotents are mapped to idempotents by homomorphisms.
\end{proof}

Part (1) of Lemma~\ref{lem:po} leads us to dub $\leq$ the {\em natural partial order} on $S$.
If a partial order is studied in relation to an inverse semigroup, then it is always this one.

Part (2) of Lemma~\ref{lem:po}  needs to be highlighted since readers familiar with lattice-ordered groups
might have been expecting something different.

Part (4) of Lemma~\ref{lem:po} tells us that when the natural partial order is restricted to the semilattice of idempotents
we get back the usual ordering on the idempotents that we have already defined.

\begin{example}{\em The natural partial order in the symmetric inverse monoids $\mathcal{I}(X)$
is precisely the usual restriction order on partial functions.
Observe that by the Wagner-Preston representation theorem Theorem~\ref{them:W-P},
we have that $a \leq b$ if and only if $\lambda_{a} \subseteq \lambda_{b}$.}
\end{example}

The natural partial order is used to define a class of inverse monoids.
We say that an inverse monoid $S$ is {\em factorizable} if for each element $s \in S$ there exists a unit $g$
such that $s \leq g$,

\begin{example}{\em The symmetric inverse monoids $\mathcal{I}(X)$ are factorizable if $X$ is finite.
Let $f \in \mathcal{I}(X)$.
Then not only do $\mbox{dom}(f)$ and $\mbox{ran} (f)$ have the same cardinality
but so too do the sets $X \setminus \mbox{dom}(f)$ and  $X \setminus \mbox{ran}(f)$.
Choose any bijection $g$ from $X \setminus \mbox{dom}(f)$ to  $X \setminus \mbox{ran}(f)$.
Then $f \cup g$ is a bijection and so is an element of the group of units of $\mathcal{I}(X)$.
However, there are many choices for $g$ and so there are many units that extend $f$.
It is precisely this lack of uniqueness in the unit that implies that we cannot reduce
inverse monoid theory to group theory.}
\end{example}

Our next result tells us that the partial order encodes how far from being a group an inverse semigroup is.

\begin{lemma}\label{lem:group-po} 
An inverse semigroup is a group if and only if the natural partial order is the equality relation.
\end{lemma}
\begin{proof} Let $S$ be an inverse semigroup in which the natural partial order is equality.
If $e$ and $f$ are any idempotents then $ef \leq e,f$ and so $e = f$.
It follows that there is exactly one idempotent. 
We deduce that $S$ is a group by Lemma~\ref{lem:unique-idempotents}.
The proof of the converse is immediate.
\end{proof}

The natural partial order is very important in studying inverse semigroups.
For this reason, it is appropriate here to introduce some terminology and notation from the theory of 
partially ordered sets.

In any partially ordered set $(X,\leq)$, a subset $Y \subseteq X$ is said to be an {\em order ideal}
if $x \leq y \in Y$ implies that $x \in Y$.
More generally, if $Y$ is any subset of $X$ then define
$$Y^{\downarrow} = \{x \in X \colon x \leq y \text{ for some } y \in Y \}.$$
This is the {\em order ideal generated by $Y$}.
If $y \in X$ then we denote $\{y \}^{\downarrow}$ by $y^{\downarrow}$
and call it the {\em principal order ideal generated by $y$}.
If $Y$ is any subset of $X$, define 
$$Y^{\uparrow} = \{x \in X \colon x \geq y \text{ for some } y \in Y \}.$$
If $Y = \{ y\}$ we denote $\{ y\}^{\uparrow}$ by $y^{\uparrow}$.
If $P$ and $Q$ are partially ordered sets then a function $\theta \colon X \rightarrow Y$ is said
to be {\em isotone} if $x \leq y$ in $X$ implies that $\theta (x) \leq \theta (y)$ in $Y$.
An {\em order-isomorphism} between two partially ordered sets is a bijective isotone function 
whose inverse is also isotone.

Part (5) of Lemma~\ref{lem:po} tells us that the semilattice of idempotents of an inverse semigroup $S$ is an order ideal in $S$
with respect to the natural partial order.

Part (6) of Lemma~\ref{lem:po} tells us that homomorphisms between inverse semigroups are isotone.

Idempotents and non-idempotents are closely related in an inverse semigroup.

\begin{lemma}\label{lem:trump} Let $S$ be an inverse semigroup.
Then there is an order-isomorphism between the set $a^{\downarrow}$ and the set $\mathbf{d}(a)^{\downarrow}$ (and, likewise, with the set $\mathbf{r}(a)^{\downarrow}$).
\end{lemma}
\begin{proof} Define a map from $a^{\downarrow}$ to $\mathbf{d}(a)^{\downarrow}$
by $b \mapsto \mathbf{d}(b)$.
This is well-defined since if $b \leq a$ then $\mathbf{d}(b) \leq \mathbf{d}(a)$.
It is isotone since if $b_{2} \leq b_{1} \leq a$ then $\mathbf{d}(b_{2}) \leq \mathbf{d}(b_{1}) \leq \mathbf{d}(a)$.
From the definition of the natural partial order,
it is immediate that this map is injective.
We define a map from  $\mathbf{d}(a)^{\downarrow}$ to  $a^{\downarrow}$ 
by $e \mapsto ae$.
It is routine to check that this is well-defined and isotone.
These two maps are mutually inverse.
It follows that the partially ordered sets are order-isomorphic.
\end{proof}

Looking below an idempotent we see only idempotents, but what happens if we look up?
The answer is that we don't necessarily see only idempotents.
The symmetric inverse monoid is an example.
This leads us to the following definition.
An inverse semigroup $S$ is said to be {\em $E$-unitary} if $e \leq s$, where $e$ is an idempotent, implies that $s$ is an idempotent.
An inverse semigroup with zero $S$ is said to be {\em $E^{\ast}$-unitary}\footnote{The term {\em $0$-$E$-unitary} is also used.} if 
$0 \neq e \leq s$ where $e$ is an idempotent implies that $s$ is an idempotent.

\begin{lemma} Let $S$ be an inverse semigroup with zero.
Then it is $E$-unitary if and only if it is a meet semilattice.
\end{lemma}
\begin{proof} Suppose that $S$ is $E$-unitary.
If $a \in S$ then $0 \leq a$.
But $0$ is an idempotent, and so $a$ is an idempotent.
The proof of the converse is imediate.
\end{proof}

The above lemma explains why we have made two definitions above, depending on whether the inverse semigroup does not or does have a zero.
This bifurcation between inverse semigroups-without-zero and inverse semigroups-with-zero permeates the subject.

There was a time when the study of $E$-unitary inverse semigroups was the centre of attention.
The two papers by Don McAlister \cite{M1, M2} are probably the most significant in that they describe the structure of $E$-unitary
inverse semigroups in terms of simpler building blocks and relate them to arbitrary inverse semigroups. 
The theory of $E^{\ast}$-unitary inverse semigroups has also been pursued.
There are both interesting examples of such inverse semigroups and the analogue of McAlister's theory can be developed in the case of the so-called
{\em strongly $E^{\ast}$-unitary inverse semigroups}.
See \cite{BFFG, Lawson1999, Lawson2002}, for example.

The following lemma tells us that in an inverse semigroup,
there is a relationship between two elements that have a common upper bound.

\begin{lemma}\label{lem:bounded-above}
Let $S$ be an inverse semigroup
and suppose that 
$a,b \leq c$.
then $a^{-1}b$ and $ab^{-1}$ are idempotents.
\end{lemma}
\begin{proof} By part (2) of Lemma~\ref{lem:po}, we have that $a^{-1},b^{-1} \leq c^{-1}$.
By part (3) of Lemma~\ref{lem:po}, we have that $a^{-1}b \leq c^{-1}c$ and $ab^{-1} \leq cc^{-1}$.
It follows by part (5) of Lemma~\ref{lem:po}, that both $a^{-1}b$ and $ab^{-1}$ are idempotents.
\end{proof}

Lemma~\ref{lem:bounded-above} leads us to the following definition.
Let $a,b \in S$, an inverse semigroup.
Define $a \sim b$ if and only if $a^{-1}b,ab^{-1} \in \mathsf{E}(S)$.
This is called the {\em compatibility relation}.
If $a \sim b$ and if their least upper bound exists, we denote it by $a \vee b$
and call it the {\em join} of $a$ and $b$.
A subset of an inverse semigroup is said to be {\em compatible} if the elements are pairwise compatible.
For example, for each element $a \in S$ in an inverse semigroup,
the set $a^{\downarrow}$ is compatible.

In an inverse semigroup with zero, there is a refinement of the compatibility relation which is important.
In such a semigroup, a pair of idempotents $e$ and $f$ is said to be {\em orthogonal}, denoted by $e \perp f$,
if and only if $ef = 0$.
We define an arbitrary pair of elements $a$ and $b$ to be {\em orthogonal}, denoted by $a \perp b$,
precisely when $\mathbf{d}(a) \perp \mathbf{d}(b)$ and $\mathbf{r}(a) \perp \mathbf{r}(b)$
You can easily check that $a \perp b$ precisely when  $a^{-1}b = 0 = ab^{-1}$.
We call $\perp$ the {\em orthogonality relation}. 
Observe that $a \perp b$ implies that $a \sim b$.
If an orthogonal subset has a least upper bound then it is said to have an {\em orthogonal join}.

\begin{example}{\em Let $f,g \in \mathcal{I}(Y)$.
Then $f \sim g$ if and only if $f \cup g$ is a partial bijection,
and $f \perp g$ if and only if the domain of $f$ is disjoint from the domain of $g$
and the range of $f$ is disjoint from the range of $g$.}
\end{example}

\begin{lemma}\label{lem:gill} Let $S$ be an inverse semigroup with zero.
If  $a \perp b$ and $c \in S$ then $ac \perp bc$ and $ca \perp cb$.
\end{lemma}
\begin{proof} We prove that $ac \perp bc$; the proof of the other case is similar.
It is routine to check that  $\mathbf{d}(ac)\mathbf{d}(bc) = 0$.
The result now follows by symmetry.
\end{proof}

The compatibility relation is reflexive and symmetric but not, in general, transitive, as the symmetric inverse monoid shows.
However, we do have an exact criterion for when the compatibility relation is transitive.

\begin{proposition}\label{prop:Eunitary} 
The compatibility relation is transitive if and only if the semigroup is $E$-unitary.
\end{proposition}
\begin{proof}
Suppose that $\sim$ is transitive.  
Let $e \leq  s$, where $e$  is  an  idempotent.  
Then $se^{-1}$ is an idempotent because $e = se = se^{-1}$, 
and $s^{-1}e$ is an idempotent because $s^{-1}e \leq  s^{-1}s$. 
Thus $s \sim e$.  
Clearly $e \sim s^{-1}s$, and so, by our assumption that
the compatibility relation is transitive,
we have that $s \sim s^{-1}s$.  
But $s(s^{-1}s)^{-1} = s$, so that $s$ is an idempotent.

Conversely, suppose that $S$ is $E$-unitary 
and that $s \sim t$ and $t \sim u$.
Clearly $(s^{-1}t)(t^{-1}u)$ is an idempotent and
$$(s^{-1}t)(t^{-1}u) = s^{-1}(tt^{-1})u \leq s^{-1}u.$$
But $S$ is $E$-unitary and so $s^{-1}u$ is an idempotent.
Similarly, $su^{-1}$ is an idempotent.
Hence $s \sim u$.
\end{proof}

There is a connection between compatible elements and certain kinds of meets.
We shall examine meets in greater generality later on in this section.

\begin{lemma}\label{lem:compatibility-meets} Let $S$ be an inverse semigroup.
\begin{enumerate}
\item $s \sim t$ if and only if $s \wedge t$ exists and 
$\dom (s \wedge t) = \dom (s) \wedge \dom (t)$
and
$\ran (s \wedge t) = \ran (s) \wedge \ran (t)$.
\item If $s \sim t$ then 
$$s \wedge t = t \mathbf{d}(s) = s \mathbf{d}(t) = \mathbf{r}(s)t = \mathbf{r}(t)s.$$
\end{enumerate}
\end{lemma}
\begin{proof}
(1) We prove first that $st^{-1}$ is an idempotent
if and only if 
the greatest  lower bound $s \wedge  t$ of $s$ and $t$ exists and 
$(s \wedge  t)^{-1}(s \wedge  t) = s^{-1}st^{-1}t$.  
The full result then follows by the dual argument.
Suppose that $st^{-1}$ is an idempotent. 
Put $z = st^{-1}t$.  
Then $z \leq  s$ and $z \leq  t$, since $st^{-1}$ is an idempotent.  
Let $w \leq  s,t$.  
Then $w^{-1}w \leq  t^{-1}t$ and so $w \leq  st^{-1}t = z$.  
Hence $z = s \wedge  t$.  
Also
$$z^{-1}z = (st^{-1}t)^{-1}(st^{-1}t) 
= t^{-1}ts^{-1}st^{-1}t = s^{-1}st^{-1}t.$$

Conversely, suppose that $s \wedge t$ exists and 
$(s \wedge  t)^{-1}(s \wedge  t) = s^{-1}st^{-1}t$.
Put $z = s \wedge t$.
Then $z = sz^{-1}z$ and $z = tz^{-1}z$.
Thus $sz^{-1}z = tz^{-1}z$, and so
$st^{-1}t = ts^{-1}s$.
Hence $st^{-1} = ts^{-1}st^{-1}$, which is an idempotent.

(2) We shall prove that $s \wedge t = t \mathbf{d}(s)$ since the other equalities follow by symmetry.
Observe that $t \mathbf{d}(s) = ts^{-1}s \leq s,t$ since $ts^{-1} = (st^{-1})^{-1}$ is an idempotent.
Suppose that $x \leq s,t$. Then $x = xx^{-1}x \leq ts^{-1}s$.
It follows that $s \wedge t = t \mathbf{d}(s)$.
\end{proof}

The following result is useful since it enables us to deduce that two elements
are equal from apparently weaker conditions.

\begin{lemma}\label{lem:karen} Let $S$ be an inverse semigroup.
If $a \sim b$ and $\mathbf{d}(a) = \mathbf{d}(b)$ (respectively, $\mathbf{r}(a) = \mathbf{r}(b)$) then $a = b$.
\end{lemma}
\begin{proof}
Suppose that $a \sim b$ then by Lemma~\ref{lem:compatibility-meets} 
the meet $a \wedge b$ exists and $\mathbf{d}(a \wedge b) = \mathbf{d}(a)\mathbf{d}(b)$.
By assumption,  $\mathbf{d}(a) = \mathbf{d}(b)$.
Thus $a \wedge b \leq a$ and $\mathbf{d}(a \wedge b) = \mathbf{d}(a)$.
It follows that $a \wedge b = a$.
Similarly, $a \wedge b = b$.
We have therefore proved that $a = b$.
\end{proof}

Inverse semigroups generalize groups, but we can also construct groups from inverse semigroups.
The idea is this.
Groups are abstract versions of groups of bijections,
and bijections can be constructed by glueing together compatible sets of partial bijections.
Thus, we could construct groups from inverse semigroups by glueing together suitable compatible subsets.
We show first how to construct groups from arbitrary inverse semigroups.
The motivation for how to do this comes from Lemma~\ref{lem:group-po}, which tells us that groups are those inverse semigroups
in which the natural partial order is equality,
and part (6) of Lemma~\ref{lem:po}, which tells us that homomorphisms between inverse semigroups are isotone. 
These two results lead us to make the following definition.
On an inverse semigroup $S$, with $s,t \in S$, define the relation $\sigma$ by 
$s \, \sigma \, t$
if and only if there is an element $u$ such that
$u \leq  s,t$.

\begin{theorem}\label{them:marin} Let $S$ be an inverse semigroup. 
\begin{enumerate}

\item $\sigma $ is a congruence on $S$.

\item $S/\sigma $ is a group. 

\item If $\rho$ is any congruence on $S$ such that $S/\rho$ 
is a group then $\sigma  \subseteq  \rho $. 

\end{enumerate}
\end{theorem}
\begin{proof}
(1)
We begin by showing that $\sigma $ is an equivalence relation.  
Reflexivity and symmetry are immediate.  
To prove transitivity, let $(a,b),(b,c) \in \sigma$. 
Then there exist elements
$u,v \in S$ such that $u \leq a,b$ and $v \leq b,c$.
Thus $u,v \leq b$.
The set $b^{\downarrow}$ is a compatible subset and so $u \wedge v$ exists
by Lemma~\ref{lem:compatibility-meets}.
But $u \wedge v \leq a,c$ and so $(a,c) \in \sigma$.
The fact that $\sigma$ is a congruence 
follows from the fact that the natural 
partial order is compatible with the multiplication.  

(2) Clearly, all idempotents are contained in a single 
$\sigma$-class (possibly with non-idempotent elements).  
Consequently, $S/\sigma $ is an inverse semigroup with a single idempotent.  
Thus $S/\sigma $ is a group by Lemma~\ref{lem:group-po}. 

(3) Let $\rho $ be any congruence such that $S/\rho $ is a group.  
Let $(a,b) \in  \sigma$.  
Then $z \leq  a,b$, for some $z$, by definition.  
Hence $\rho (z) \leq  \rho (a),\rho (b)$ since homomomorphisms between inverse semigroups are isotone.
But $S/\rho $ is a group and so its natural partial order is equality.  
Hence $\rho (a) = \rho (b)$.
\end{proof}

In the light of Theorem~\ref{them:marin},
it is natural to call the congruence $\sigma$ the {\em minimum group congruence}
and the group $S/\sigma$ the {\em maximum group image} of $S$.

A homomorphism $\theta \colon S \rightarrow T$ is said to be {\em idempotent-pure}
if $\theta (a)$ an idempotent implies that $a$ is an idempotent.

\begin{lemma}\label{lem:splott} Let $S$ be an inverse semigroup.
\begin{enumerate}
\item $\sim\, \subseteq \sigma$.
\item The congruence $\rho$ is idempotent-pure if and only if $\rho\, \subseteq \, \sim$.
\end{enumerate}
\end{lemma}
\begin{proof} (1) Suppose that $a \sim b$.
Then $a \wedge b$ exists by Lemma~\ref{lem:compatibility-meets}.
It follows that $a \, \sigma \, b$.

(2) Suppose that $\rho$ is idempotent-pure and that $a \, \rho \, b$.
Then $ab^{-1} \, \rho \, bb^{-1}$.
But $\rho$ is idempotent-pure and so $ab^{-1}$ is an idempotent.
Similarly, $a^{-1}b$ is an idempotent.
We have proved that $a \sim b$.
Conversely, suppose that  $\rho$ is a congruence such that $\rho\, \subseteq \, \sim$
and $a \, \rho\, e$, where $e$ is an idempotent.
Observe that $a^{-1} \, \rho\, e$.
Thus $aa^{-1} \, \rho \, e$.
It follows that $aa^{-1} \, \rho \, a$.
Thus $(aa^{-1})a$ is an idempotent and so $a$ is an idempotent.
\end{proof}

Inverse semigroups and their homomorphisms form a category (of structures).
The category of groups and their homomorphisms is a subcategory.
The properties of the minimum group congruence lead naturally to
the following result on the category of inverse semigroups.

\begin{proposition}\label{prop:groups-reflexive} 
The category of groups is a reflective subcategory of
the category of inverse semigroups.
\end{proposition}
\begin{proof}
Let $S$ be an inverse semigroup and let
$\sigma^{\natural} \colon S \rightarrow S/\sigma$ be the associated natural
homomorphism.
Let $\theta \colon S \rightarrow G$ be a homomorphism to a group $G$.
Then $\mbox{\rm ker}(\theta)$ is a group congruence on $S$ and so
$\sigma \subseteq \mbox{\rm ker}(\theta)$ by Theorem~\ref{them:marin}.
Thus by standard semigroup theory, there is a unique homomorphism $\theta^{\ast}$ from $S/\sigma$ to $G$ 
such that $\theta = \theta^{\ast} \sigma^{\natural}$.
\end{proof}

It follows by standard category theory, such as \cite[Chapter IV, Section 3]{Maclane},
that there is a functor from the category of 
inverse semigroups to the category of groups which takes each 
inverse semigroup $S$ to $S/\sigma$
and if $\theta \colon S \rightarrow  T$ is a homomorphism 
of inverse semigroups then the function  
$\psi \colon  S/\sigma  \rightarrow  T/\sigma $  
defined by $\psi(\sigma (s)) = \sigma (\theta (s))$  
is the corresponding group homomorphism (this can be checked directly).

There is another characterization of $E$-unitary inverse semigroups
which is interesting in this context.

\begin{lemma}\label{lem:another-characterization} Let $S$ be an inverse semigroup.  
Then the following conditions are equivalent: 
\begin{enumerate}

\item $S$ is $E$-unitary. 
 
\item $\sim \, = \sigma $. 
 
\item $\sigma$ is idempotent-pure.

\item $\sigma (e) = \mathsf{E}(S)$ for any idempotent $e$.

\end{enumerate}
\end{lemma}
\begin{proof}
(1)$\Rightarrow$(2).  
By part (1) of Lemma~\ref{lem:splott}, 
the compatibility relation is contained in $\sigma$.
Let $(a,b) \in  \sigma $.  
Then $z \leq  a,b$ for some $z$.  
It follows that $z^{-1}z \leq  a^{-1}b$ and $zz^{-1} \leq  ab^{-1}$.  
But $S$ is $E$-unitary and so $a^{-1}b$ and $ab^{-1}$ are both idempotents.  
Hence $a \sim b$. 

(2)$\Rightarrow$(3). 
By part (2) of Lemma~\ref{lem:splott}, 
 a congruence is idempotent pure precisely
when it is contained in the compatibility relation.

(3) $\Rightarrow$ (4).  
This is immediate from the definition of an idempotent-pure congruence.

(4) $\Rightarrow (1)$
Suppose that $e \leq a$ where $e$ is an idempotent.
Then $(e,a) \in \sigma$.
But by our assumption, the element $a$ is an idempotent.
\end{proof} 

For inverse semigroups with zero the minimum group congruence is not very interesting
since the group degenerates to the trivial group.
The following example shows one concrete way to deal with this issue.

\begin{example}{\em Let $G$ be a group and let $S$ be the inverse monoid of all isomorphisms between
the subgroups of $G$. The semilattice of idempotents is isomorphic to the partially ordered set of subgroups of $G$.
The trivial group is a subgroup of every group. 
It follows that the isomorphism from the trivial subgroup to itself is the zero of $S$.
Suppose now that $G$ is infinite.
Consider the inverse subsemigroup $\Omega (G)$ of $S$ which consists of all the isomorphisms 
between the subgroups of $G$ of finite index.
The group $\mbox{Comm}(G) = \Omega (G)/\sigma$ is called the {\em abstract commensurator} of $G$ \cite{BB}.
See also \cite{Nek2002}.
The elements of $\mbox{Comm}(G)$ are `hidden symmetries' to use the terminology of Farb and Weinberger \cite{FW2004}.
}
\end{example}

The above example shows that we may form groups from inverse semigroups by looking at
the `large' elements of the inverse semigroup (and so excluding the zero).
Here is another example.
Let $S$ be an inverse semigroup.
A non-zero idempotent $e$ of an inverse semigroup $S$ is said to be {\em essential} if $ef \neq 0$ for all non-zero
idempotents $f$ of $S$.
An element $s$ is said to be {\em essential} if both $s^{-1}s$ and $ss^{-1}$ are essential.
Essential elements were first defined in \cite{Birget} and applied to construct groups in \cite{Lawson2007, Lawson2007b}.
Denote by $S^{e}$ the set of all essential elements of $S$.
We regard the elements of $S^{e}$ as being `large' elements of $S$. 
We call $S^{e}$ the {\em essential part of $S$.}

\begin{lemma}\label{lem:essential-elements}
Let $S$ be an inverse semigroup.
Then, if non-empty, $S^{e}$ is an inverse subsemigroup of $S$.
\end{lemma}
\begin{proof} It is clear that $S^{e}$ is closed under inverse.
Let $a,b \in S^{e}$.
We prove that $\mathbf{d}(ab)$ is essential;
the proof that $\mathbf{r}(ab)$ is essential is similar.
Observe, first, that $b^{-1}a^{-1}ab \neq 0$ since both $bb^{-1}$ and $a^{-1}a$ are essential.
Let $e$ be a non-zero idempotent. 
We calculate $\mathbf{d}(ab)e$.
Observe that $beb^{-1} \neq 0$ and is an idempotent, since it is the conjugate of an idempotent.
Thus $a^{-1}abeb^{-1} \neq 0$.
It follows that  $a^{-1}abe \neq 0$ and so  $b^{-1}a^{-1}abe \neq 0$.
\end{proof}

We expect the groups $S^{e}/\sigma$ to be interesting,
which indeed they are if we choose $S$ carefully.
See \cite{Lawson2007b,LV,LSV} for applications of the essential part of an inverse semigroup
in constructing groups.


We have seen that there is a precondition that must be satisfied in order that
a pair of elements have a join.
The same is not true for meets.
If every pair of elements of an inverse semigroup has meets we say that
it is a {\em meet-semigroup}.
Such semigroups were first studied by Leech \cite{Leech}.
There is an alternative way to characterize inverse meet-semigroups which is often useful.

\begin{lemma}\label{lem:meets} Let $S$ be an inverse semigroup.
\begin{enumerate}
\item $S$ has all binary meets if and only if
for each element $a \in S$ there is an idempotent, denoted by $\phi (a)$, such that $\phi (a) \leq a$ and 
if $e \leq a$, where $e$ is an idempotent, then $e \leq \phi (a)$.
Thus, $\phi (a)$ is the largest idempotent below $a$.
\item The map $\phi \colon S \rightarrow \mathsf{E}(S)$, defined in part (1) above, is an order-preserving idempotent function with $\mathsf{E}(S)$ as its fixed-point set
such that $\phi (ae) = \phi(a)e$ and $\phi(ea) = e \phi (a)$ for all $e \in \mathsf{E}(S)$.
\end{enumerate}
\end{lemma}
\begin{proof} (1) Suppose first that $S$ has all binary meets.
For each $a \in S$ define $\phi (a) = a \wedge \mathbf{d}(a)$.
Clearly, $\phi (a)$ is an idempotent (because it is beneath an idempotent).
Let $e$ be an idempotent such that $e \leq a$.
Then $e \leq \mathbf{d}(a)$ and so $e \leq \phi (a)$.
We have therefore proved that if all binary meet exist, then the function $\phi$ exists.
Conversely, suppose that such a function $\phi$ exists.
Let $a,b \in S$ and consider the element $\phi (ab^{-1})b$.
Clearly, $\phi (ab^{-1})b \leq b$.
But, by definition, $\phi (ab^{-1}) \leq ab^{-1}$ and so  $\phi (ab^{-1})b \leq ab^{-1}b \leq a$.
Let $c \leq a,b$.
Then $cc^{-1} \leq ab^{-1}$ and so $cc^{-1} \leq \phi (ab^{-1})$.
Now $c \leq b$ and so $c = (cc^{-1})c \leq \phi (ab^{-1})b$.
It follows that $a \wedge b =  \phi (ab^{-1})b$.

(2) It is easy to prove that $\phi$ is order-preserving.
We prove that $\phi (ae) = \phi(a)e$ for all $e \in \mathsf{E}(S)$.
By definition $\phi (a) \leq a$.
Thus $\phi (a)e \leq ae$.
But $\phi (a)e$ is an idempotent.
It follows that $\phi (a)e \leq \phi (ae)$.
We have that $\phi (ae) \leq ae \leq a$.
But $\phi (ae)$ is an idempotent.
It follows that $\phi (ae) \leq \phi(a)$ and so $\phi (ae)e  \leq \phi (a)e$.
But $\phi (ae) \leq ae$, and so $\phi (ae)e = \phi (ae)$.
Thus $\phi (ae) \leq \phi (a)e$.
We have therefore proved that $\phi (ae) = \phi (a)e$. 
\end{proof}

The function $\phi$ is called a {\em fixed-point operator}.

\begin{example}{\em We can deduce right away that the symmetric inverse monoids are meet-monoids
because if $f \in \mathcal{I}(X)$ then we can define the idempotent
$1_{A}$ where $A$ is the set of all points that $f$ fixes
(and if $f$ does not fix any points then $A = \varnothing$.) 
}
\end{example}

The $E^{\ast}$-unitary semigroups also enjoy a property that is more significant than it looks.
The following was first noted by \cite{Lenz}.

\begin{proposition} 
An $E^{\ast}$-unitary inverse semigroup has meets of all pairs of elements.
\end{proposition}
\begin{proof}
Let $s$ and $t$ be any pair of elements.
Suppose that there exists a non-zero element $u$ such that $u \leq s,t$.
Then $uu^{-1} \leq st^{-1}$ and $uu^{-1}$ is a non-zero idempotent.
Thus $st^{-1}$ is an idempotent.
Similarly $s^{-1}t$ is an idempotent.
It follows that $s \wedge t$ exists by Lemma~\ref{lem:compatibility-meets}.
If the only element below $s$ and $t$ is $0$ then $s \wedge t = 0$.
\end{proof}

To conclude this section, we state another deep result about finite inverse monoids.
We need a definition first.
An inverse monoid $S$ is said to be {\em $F$-inverse} if every $\sigma$-class contains a greatest element.

\begin{lemma}\label{lem:Finverse} 
Every $F$-inverse monoid is $E$-unitary.
\end{lemma}
\begin{proof} Let $S$ be an $F$-inverse monoid.
Suppose that $e \leq a$.
Then $\sigma (e) = \sigma (a)$.
Each $\sigma$-class contains a maximum element.
Since we are working in a monoid, $\sigma (e) = \sigma (1)$.
Let the maximum element in $\sigma (1)$ be $x$.
Then $1 \leq x$.
It follows that $1 = x1 = x$.
Thus the maximum element of $\sigma (1)$ is $1$ itself.
Thus $a \leq 1$ and so $a$ is an idempotent.
\end{proof}

You can find out a lot more about $F$-inverse monoids, here \cite{AKS}.
In \cite[Theorem 2.7]{ABO}, the following deep theorem is proved
and the background to it explained.

\begin{theorem} Every finite inverse monoid
has a finite $F$-inverse monoid cover.
\end{theorem}

The implications of this theorem for finite inverse monoid theory are explained in \cite{Lawson1994}, 
but it is enough to say that this is another example of a beautiful piece of mathematics
and is related to Ash's deep result Theorem~\ref{them:ash-deep-theorem}.

\section{Inverse semigroups as non-commutative lattices}

It is only a slight exaggeration to say that in the four decades after inverse semigroups
were introduced the main focus of researchers was on the purely algebraic properties of inverse semigroups.\footnote{This is a tendancy rather than an absolute. 
For example, the papers \cite{Rinow} and \cite{Schein} deal with the natural partial order.}
Subsequently, the properties of the natural partial order have come to the fore.
In this section, we shall study inverse semigroups with respect to this natural partial order.
From this point of view, inverse semigroups can themselves be regarded as `non-commutative meet semilattices'.
However, we shall find it fruitful to regard certain inverse semigroups
as being `non-commutative lattices' of various kinds, always with the proviso
that, because of Lemma~\ref{lem:bounded-above}, the join will not always be defined.
Meets of idempotents will always usually be denoted by concatenation.
Specifically, we shall take as our points of departure the various classes of lattice (each with top $1$ and bottom $0$):
a {\em frames} are the complete infinitely distributive lattices;
{\em distributive lattices} have all binary meets and binary joins
with binary meets distributing over binary joins and vice-versa;
{\em Boolean algebras} are those distibutive lattices in which each element $x$ has a {\em complement} $\bar{x}$
such that $x\bar{x} = 0$ and $x \vee \bar{x} = 1$.
Observe that homomorphisms of distributive lattices
automatically preserve complements.
Boolean algebras are particularly important
so it will be useful to have a way of defining them
which involves only products and complements.
For the standard axioms for a Boolean algebra see \cite[Chapter 2]{LawsonLog}.
The following lemma contains some axioms due to Frink  \cite{Frink}.
If you have trouble proving that this really is a Boolean algebra, see \cite{Pad}.

\begin{lemma}\label{lem:Boolean algebras} Consider the following structure
$(B, \cdot, a \mapsto \bar{a}, 0)$
where $(B,\cdot)$ is a commutative band and $ab = a$ if and only if $a\bar{b} = 0$.
Define $a + b = \overline{(\bar{a} \cdot \bar{b})}$.
Then $(B,\cdot,+,0,\bar{0})$ is a Boolean algebra.
\end{lemma}

We call these {\em Frink's axioms for a Boolean algebra}.

We begin with frames.
The lattice of open sets of a topological space is an example of a frame
The theory of frames can be viewed as the theory of topological spaces in which the open sets, and not the points, are taken as primary.
As well as being an interesting theory in its own right \cite{J} with important applications, it is also a key ingredient in topos theory \cite{MM}. 
Johnstone discusses the origins of frame theory in his book on this subject \cite[Chapter II]{J}.
One sentence is significant for the goals of this section.
He writes on page~76:
\begin{quote}
It was Ehresmann \ldots and his student B\'enabou \ldots who first took the decisive step in regarding complete
Heyting algebras as `generalized topological spaces'.
\end{quote}
However, Johnstone does not say {\em why} Ehresmann was led to his frame-theoretic viewpoint of topological spaces.
In fact, it was Ehresmann's paper \cite{Ehresmann1957}, 
which we have cited above as being one of the origins of inverse semigroup theory,
which led to the theory of frames.
Ehresmann was interested in pseudogroups of transformations.
Amongst those are the full transformation pseudogroups $\mathcal{I}(X,\tau)$
of homeomorphisms between the open subsets of $\tau$.
The idempotents of such pseudogroups are the identity functions defined on the 
open subsets of $\tau$.
Of course, these form frames.
More generally, 
we define a {\em pseudogroup} to be an inverse semigroup in which all compatible subsets have joins
and in which multiplication distributes over such joins.
The semilattice of idempotents of a pseudogroup is a frame and so we can regard
pseudogroups as being non-commutative generalizations of frames.
Using this language, we can say that Schein \cite{Schein} proved that associated with every inverse semigroup 
is a universal pseudogroup.
Pseudogroups themselves are the subject of \cite{Resende2007}. 

We can continue in this vein, but consider instead only finite joins.
We say that an inverse monoid is {\em distributive} if it has all finite joins
and multiplication distributes over such joins.
The semilattice of idempotents of a distributive inverse monoid is a distributive lattice.
We can therefore regard distributive inverse monoids as being non-commutative distributive lattices.
Distributive inverse monoids were first studied in \cite{KM}.

An inverse monoid is said to be {\em Boolean} if it is distributive and its semilattice of idempotents
is in fact a Boolean algebra.
We can therefore regard Boolean inverse monoids as being non-commutative Boolean algebras.
Boolean inverse monoids have shown themselves to be particularly interesting.

The relationships between these various classes of inverse semigroup is described in \cite{LL}.
Boolean inverse monoids were introduced in \cite{Lawson2010},
and the fact that associated with every inverse semigroup is a universal Boolean inverse semigroup
is proved in \cite{Lawson2020a}.
Associated with every Boolean inverse monoid is its {\em type monoid}:
this is always an abelian refinement monoid \cite{KLLR}.
The type monoid is then studied in great detail in \cite{Wehrung}.
The following table summarizes this whole approach to inverse semigroups:

\vspace{0.5cm}
\begin{center}
\begin{tabular}{|c||c|}\hline
{\bf Commutative} & {\bf Non-commutative}  \\ \hline \hline
{\small Meet semilattice} & {\small Inverse semigroup} \\ \hline
{\small Frame} & {\small Pseudogroup}  \\ \hline
{\small Distributive lattice} & {\small Distributive inverse monoid}  \\ \hline 
{\small Boolean algebra} & {\small Boolean inverse monoid} \\ \hline
\end{tabular}
\end{center}
\vspace{0.5cm}

In the remainder of this section, we shall be particularly interested in Boolean inverse monoids,
but we shall prove some results in greater generality.
Our first results hold in any inverse semigroup

\begin{lemma}\label{lem:stuff} Let $S$ be an inverse semigroup.
\begin{enumerate}

\item Here, we shall need the fact that the inverse semigroup has a zero.
Suppose that $a,b \leq c$.
If 
$\mathbf{d}(a) \perp \mathbf{d}(b)$ 
then
$\mathbf{r}(a) \perp \mathbf{r}(b)$.

\item If $a \wedge b$ exists then $c(a \wedge b) = ca \wedge cb$ and $(a \wedge b)c = ac \wedge bc$.

\end{enumerate}
\end{lemma}
\begin{proof} (1) We are given that $a = c\mathbf{d}(a)$ and $b = c \mathbf{d}(b)$.
Using these equations and the fact that idempotents commute,
it is easy to check that $\mathbf{r}(a)\mathbf{r}(b) = 0$.

(2) We are given that $a \wedge b$ exists.
We prove that $ca \wedge cb$ exists and that  $(a \wedge b)c = ac \wedge bc$.
The proof of the other case is similar.
Observe that $c(a \wedge b) \leq ca,cb$.
Now let $y \leq ca,cb$.
Then $c^{-1}y \leq c^{-1}ca, c^{-1}cb$ and so $c^{-1}y \leq a,b$.
It follows that $c^{-1}y \leq a \wedge b$.
Thus $cc^{-1}y \leq c(a \wedge b)$.
But $y \leq ca$ and so $y = ca\mathbf{d}(y)$.
It follows that $cc^{-1}y = y$.
Thus $y \leq ca, cb$ implies that $y \leq c(a \wedge b)$.
It follows that $c(a \wedge b) = ca \wedge cb$
\end{proof}

We now specialize to distributive inverse monoids.
Result (2) below is what remains of one of the distributive laws.

\begin{lemma}\label{lem:domains} Let $S$ be a distributive inverse monoid.
\begin{enumerate}

\item Suppose that $a \sim b$.
Then
$\mathbf{d}{(a \vee b)} = \mathbf{d}(a) \vee \mathbf{d}(b)$ 
and 
$\mathbf{r}{(a \vee b)} = \mathbf{r}(a) \vee \mathbf{r}(b)$ 

\item Suppose that both $a \vee b$ and $c \wedge (a \vee b)$ exist.
Then both $c \wedge a$ and $c \wedge b$ exist, the join $(c \wedge a) \vee (c \wedge b)$ exists
and  $c \wedge (a \vee b) = (c \wedge a) \vee (c \wedge b)$.

\end{enumerate}
\end{lemma}
\begin{proof} (1) We prove the first result; the proof of the second is similar.
We calculate $(a \vee b)^{-1}(a \vee b)$.
We first use the fact that $a \mapsto a^{-1}$ is an order-isomorphism.
It follows that 
$(a \vee b)^{-1}(a \vee b) 
=
(a^{-1} \vee b^{-1})(a \vee b)$.
Now multiply out to get
$$a^{-1}a \vee a^{-1}b \vee b^{-1}a \vee b^{-1}b.$$
But $a \sim b$ and so both $a^{-1}b$ and $b^{-1}a$ are idempotents.
Morover, $a^{-1}b \leq a^{-1}a$
and
$b^{-1}a \leq b^{-1}b$.
The result now follows.

(2) Let $x \leq c,a$.
Then $x \leq c \wedge (a \vee b)$. 
It follows that $x \mathbf{d}(a) \leq (c \wedge (a \vee b))\mathbf{d}(a)$. 
But $x \leq a$ implies that $x \mathbf{d}(a) = x$.
It follows that $x \leq (c \wedge (a \vee b))\mathbf{d}(a)$. 
On the other hand,
$(c \wedge (a \vee b))\mathbf{d}(a) \leq c,a$. 
It follows that
$$c \wedge a = (c \wedge (a \vee b))\mathbf{d}(a).$$
Similarly,
$$c \wedge b = (c \wedge (a \vee b))\mathbf{d}(b).$$
Observe that since $c \wedge a, c \wedge b \leq a \vee b$
it follows that $(c \wedge a) \sim (c \wedge b)$.
Thus $(c \wedge a) \vee (c \wedge b)$ exists.
Observe that
$(c \wedge a) \vee (c \wedge b) \leq c \wedge a, c \wedge b$.
It follows that  $(c \wedge a) \vee (c \wedge b) \leq c, a \vee b$.
Thus,  
$$(c \wedge a) \vee (c \wedge b) \leq c \wedge (a \vee b).$$
Let $x \leq a \vee b, c$.
Then $x = (a \vee b)\mathbf{d}(x)$. 
From $x \leq a \vee b$ we get that $x\mathbf{d}(a) \leq (a \vee b)\mathbf{d}(a)$.
Now, $(a \vee b)\mathbf{d}(a) = a$.
It follows that $x\mathbf{d}(a) \leq a$.
But $x \leq c$ and so $x\mathbf{d}(a) \leq c\mathbf{d}(a)$.
Thus we have proved that $x\mathbf{d}(a) \leq a, c\mathbf{d}(a)$.
We now apply 
Lemma~\ref{lem:stuff},
and deduce that $(a \wedge c)\mathbf{d}(a) = a \wedge c\mathbf{d}(c)$.
Thus $x\mathbf{d}(a) \leq a \wedge c\mathbf{d}(c)$.
Similarly, $x \mathbf{d}(b) \leq c \mathbf{d}(b) \wedge b$.
It follows that 
$$x \mathbf{d}(a) \vee x \mathbf{d}(b) \leq (c\mathbf{d}(a) \wedge a) \vee (c\mathbf{d}(b) \wedge b)$$
and so 
$$x\mathbf{d}(a \vee b) \leq  (c\mathbf{d}(a) \wedge a) \vee (c\mathbf{d}(b) \wedge b).$$
By the above, we deduce that
$$x \leq (c\mathbf{d}(a) \wedge a) \vee (c\mathbf{d}(b) \wedge b).$$
But
$(c\mathbf{d}(a) \wedge a) \vee (c\mathbf{d}(b) \wedge b) \leq (c \wedge a) \vee (c \wedge b)$
and the result follows.
\end{proof}

By induction, the second result above can be generalized to any number of joins.

We now specialize to Boolean inverse monoids.
Let $S$ be a Boolean inverse monoid.
If $y \leq x$, define
$$x \setminus y = x\overline{\mathbf{d}(y)}.$$

\begin{lemma}\label{lem:chicken} Let $S$ be a Boolean inverse monoid.
\begin{enumerate}
\item $\mathbf{d}(x \setminus y) = \mathbf{d}(x)\overline{\mathbf{d}(y)}$.
\item If $y \leq x$ then $y \perp (x \setminus y)$ and $x = y \vee (x \setminus y)$.
\item Suppose that $a \leq x$ is such that $a \perp y$ and $x = y \vee a$ then
$a = x \setminus y$. 
\item $\mathbf{r}(x \setminus y) = \mathbf{r}(x)\overline{\mathbf{r}(y)}$.
\end{enumerate}
\end{lemma}
\begin{proof} (1) Straightforward from the definition.  

(2) By definition $(x \setminus y) \leq x$,
and from part (1), we have that
$\mathbf{d}(x \setminus y) = \mathbf{d}(x)\overline{\mathbf{d}(y)}$.
It follows that $\mathbf{d}(y) \perp \mathbf{d}(x \setminus y)$.
We now apply Lemma~\ref{lem:stuff}, to deduce that
$y \perp (x \setminus y)$.
Clearly, $y \vee (x \setminus y) \leq x$.
But $\mathbf{d}(y \vee (x \setminus y)) = \mathbf{d}(y) \vee \mathbf{d}(x)\overline{\mathbf{d}(y)}) = \mathbf{d}(x)$.
It follows that $x = y \vee (x \setminus y)$.

(3) Observe that 
$\mathbf{d}(x) = \mathbf{d}(y) \vee \mathbf{d}(a)$
and
$\mathbf{d}(a) \perp \mathbf{d}(y)$.
It follows that $\mathbf{d}(a) = \mathbf{d}(x) \overline{\mathbf{d}(y)}$.
But $a = x \mathbf{d}(a) = x \overline{\mathbf{d}(y)} = x \setminus y$.

(4) This follows by part (3) above.
\end{proof}

A {\em morphism} between distributive inverse monoids
is a homomorphism of monoids with zero that preserves binary joins.

In working with Boolean inverse monoids, it is often much easier to calculate with orthogonal joins than arbitrary ones.
The following result provides circumstances under which we lose nothing by doing this.

\begin{lemma}\label{lem:queen}  Let $S$ be an inverse semigroup.
\begin{enumerate}
\item We make the following assumptions about $S$:
the semilattice of idempotents of $S$ forms a Boolean algebra under the natural partial order;
all finite orthogonal joins exist;
multiplication distributes over finite orthoginal joins.
Then $S$ is a Boolean inverse monoid.

\item Let $S$ and $T$ be Boolean inverse monoids.
Let $\theta \colon S \rightarrow T$ be a monoid homomorphism of monoid with zero that preserves orthogonal joins.
Then $\theta$ preserves all binary compatible joins.

\end{enumerate}
\end{lemma}
\begin{proof} (1) We have to prove that all binary joins exist and that multiplication distributes over such joins.
Let $e,f \in \mathsf{E}(S)$.
Then using the distributive law, we have that $e\bar{f} \vee f = e \vee f$.
But $e\bar{f}$ and $f$ are orthogonal.
It follows that $a(e \vee f) = a(e\bar{f} \vee f) = ae\bar{f} \vee af$. 
Observe that $(ae\bar{f} \vee af)\mathbf{d}(ae) = ae$.
It follows that $ae, af \leq ae\bar{f} \vee af$. 
On the other hand, if $ae, af \leq x$
then $ae\bar{f} \vee af \leq x$.
We have therefore proved that 
$$a(e \vee f) = ae \vee af$$
where $e$ and $f$ are any idempotents.
Suppose, now, that $a \sim b$.
Then $a \wedge b$ exists by Lemma~\ref{lem:compatibility-meets}.
We have that
$$(a \setminus a \wedge b) \vee b 
= 
a\overline{\mathbf{d}(a)\mathbf{d}(b)} \vee b$$
since, in this case, 
$\mathbf{d}(a \wedge b) = \mathbf{d}(a)\mathbf{d}(b)$
by Lemma~\ref{lem:compatibility-meets}.
But $\overline{\mathbf{d}(a)\mathbf{d}(b)} = \overline{\mathbf{d}(a)} \vee \overline{\mathbf{d}(b)}$.
We now use our result above to deduce that
$$(a \setminus a \wedge b) \vee b = a \overline{\mathbf{d}(b)} \vee b,$$
an orthogonal join.
Observe that  $(a \overline{\mathbf{d}(b)} \vee b)\mathbf{d}(a) = a$,
where we use the fact that $b \mathbf{d}(a) = a \mathbf{d}(b)$ by Lemma~\ref{lem:compatibility-meets},
our result above, and a little Boolean algebra.
We therefore have that $a,b \leq a \overline{\mathbf{d}(b)} \vee b$.
Suppose that $a,b \leq x$.
Then it is routine to check that $a\overline{\mathbf{d}(b)} \vee b \leq x$.
We have therefore proved that $a \vee b$ exists and is equal to the orthogonal join $a \overline{\mathbf{d}(b)} \vee b$.

It remains to show that multiplication distributes over binary joins.
Suppose that $a \sim b$.
It can be proved directly that that $ca \sim cb$.
We prove that $$c(a \vee b) = ca \vee cb.$$
We have that 
$$c(a \vee b) = c(a\overline{\mathbf{d}(b)} \vee b) = ca\overline{\mathbf{d}(b)} \vee cb.$$
Thus 
$$ca \leq ca\overline{\mathbf{d}(b)} \vee cb.$$
Now, 
$$(ca\overline{\mathbf{d}(b)} \vee cb)\mathbf{d}(ca) = ca \overline{\mathbf{d}(b)} \vee cb \mathbf{d}(ca).$$
We now use the fact that $ca \sim cb$ and Lemma~\ref{lem:compatibility-meets} to get
$$ca \overline{\mathbf{d}(b)} \vee cb \mathbf{d}(ca) =  ca \overline{\mathbf{d}(b)} \vee ca \mathbf{d}(cb) \leq ca.$$
It follows that 
$$ca = (ca\overline{\mathbf{d}(b)} \vee cb)\mathbf{d}(ca).$$
We deduce that $ca,cb \leq ca\overline{\mathbf{d}(b)} \vee cb$.
Suppose that $ca, cb \leq x$.
Then it is routine to check that $ca\overline{\mathbf{d}(b)} \vee cb \leq x$.
Whence $ca \vee cb =  ca\overline{\mathbf{d}(b)} \vee cb$
and so $c(a \vee b) = ca \vee cb$.

(2) Suppose that $a \sim b$.
Then $a \vee b = a \overline{\mathbf{d}(b)} \vee b$, which is an orthogonal join.
By assumption, 
$$\theta (a \vee b) =  \theta (a \overline{\mathbf{d}(b)}) \vee \theta (b) = \theta (a) \theta (\overline{\mathbf{d}(b)})) \vee \theta (b).$$
The result now follows since $\theta (\bar{e}) = \overline{\theta (e)}$ where $e$ is any idempotent.
\end{proof}

For the remainder of this section, we shall focus on Boolean inverse monoids.
As we shall prove in Theorem~\ref{them:bundestag}, 
they arise naturally as soon as we try to map inverse monoids
to the multiplicative monoids of rings, such as in representation theory.
Accordingly, we begin with some results about idempotents in rings.
Recall that all our rings will be unital.
If $e$ and $f$ are idempotents in a ring $R$ then we say they are {\em orthogonal}
if $ef = 0$ denoted by $e \perp f$.
A finite set of idempotents is {\em orthogonal} if each distinct pair of elements is orthogonal.
A sum of a finite number of orthogonal idempotents will be called an {\em orthogonal sum}.
More generally, a sum of orthogonal elements taken from an inverse submonoid will be called an {\em orthogonal sum}.

\begin{lemma}\label{lem:ring-idempotents} Let $R$ be a unital ring.
\begin{enumerate}

\item If $e$ is an idempotent in $R$, then $1 - e$ is an idempotent in $R$;
in addition, the idempotents $e$ and $1-e$ are orthogonal.

\item If $e$ and $f$ are orthogonal idempotents then $e + f$ is an idempotent; this can be generalized to
any finite sets of orthogonal idempotents.

\item If $e$ and $f$ are orthogonal idemptents then 
$$1 - (e + f) = (1 - e)(1 - f).$$ 

\item If $e$ and $f$ are commuting idempotents, then 
$$e + f - ef = f + e(1 - f),$$ 
an orthogonal sum.

\item If $e$ and $f$ are commuting idempotents, then 
$$e \circ f = e + f - ef$$ 
is an idempotent.

\item If $e$ and $f$ are commuting idempotents then 
$$(1 - e)(1 - f) = 1 - e \circ f.$$

\item If $e,f,g$ is a set of commuting idempotents, then $e$ commutes with $f \circ g$ and
$g$ commutes with $e \circ f$ and 
$$(e \circ f) \circ g = e \circ (f \circ g).$$

\item If $e_{1}, \ldots, e_{m}$ are orthogonal idempotents, then 
$$1 - \left( \sum_{i=1}^{m} e_{i} \right) = \prod_{i=1}^{m} (1 - e_{i}).$$ 

\item Let $e_{1},\ldots,e_{m}$ be a set of commuting idempotents.
Define 
$$[e_{1},\ldots, e_{m}] = (\ldots ((e_{1} \circ e_{2}) \circ e_{3}) \ldots ) \circ e_{m}.$$
Then 
$$\prod_{i=1}^{m} (1 - e_{i}) =  1 - [e_{1},\ldots, e_{m}].$$

\item Let $e_{1},\ldots,e_{m}$ be a set of commuting idempotents.
Then 
$$[e_{1},\ldots, e_{m}]$$
is defined to be
$$e_{1}(1 - e_{2}) \ldots (1 - e_{m}) + e_{2}(1 - e_{3})\ldots (1 - e_{m}) + \ldots + e_{m-1}(1 - e_{m}) + e_{m},$$
a sum of orthogonal idempotents.

\end{enumerate}
\end{lemma}
\begin{proof} The proofs of (1)--(7) are straightforward.

(8) By part (3) and induction.

(9) By part (6) and induction.

(10) By part (4) and induction.
\end{proof}

We begin with a lemma which is a special case of the result we want to prove.

\begin{lemma}\label{lem:boolean-ring} Let $E \subseteq R$, where $R$ is a unital ring.
Suppose that $0,1 \in E$ and $E$ is a commutative band under the induced multiplication from $R$.
Then there is a subset $E \subseteq B \subseteq R$
such that $B$ is a Boolean algebra with meet given by the product in the ring.
\end{lemma}
\begin{proof} If $E$ is any semilattice in $R$, 
define 
$$E' =\{e(1-e_{1}) \ldots (1 - e_{m}) \colon e,e_{1},\ldots, e_{m} \in E\} \cup E.$$
Then it is easy to check that this is a commutative band.
If $E$ is any semilattice in $R$, define $E^{\perp}$ to be the set of all sums of finite sets of orthogonal elements of $E$.
Then it is easy to check that this is a commutative band.
Put $B = (E')^{\perp}$.
This is a commutative band.
To prove that $B$ is a Boolean algebra, we use Lemma~\ref{lem:Boolean algebras}.
We claim that if $e \in B$ then $1 - e \in B$.
An element of $B$ is an orthogonal join of idempotents that belong to $E'$.
Observe that if $e_{1},\ldots, e_{m}$ are any orthogonal idempotents of $B$ then
$1 - \left( \sum_{i=1}^{m} e_{i} \right) = \prod_{i=1}^{m} (1 - e_{i})$
by part (8) of Lemma~\ref{lem:ring-idempotents}
where each $e_{i} \in E'$.
If $e \in E$ then $1 - e \in E'$.
Thus we need only concentrate on idempotents of the form 
$$e(1 - f_{1}) \ldots (1 - f_{n}) \in E'$$ 
where $f_{1},\ldots, f_{n} \in E$.
Then 
$$1 - e(1 - f_{1}) \ldots (1 - f_{n}) = (1 - e) + e[f_{1},\ldots, f_{n}],$$
using part (9) of Lemma~\ref{lem:ring-idempotents},
which is a sum of orthogonal idempotents by part (10) of Lemma~\ref{lem:ring-idempotents}.
We have therefore proved that $B$ is closed under taking complements.
Let $e,f \in B$ and suppose that $ef = e$.
Then $e(1-f) = e - ef = e - e = 0$.
On the other hand, suppose that $e(1-f) = 0$.
Then $e = ef$.
We have that $B$ satisfies Frink's axioms and so $B$ is a Boolean algebra.
\end{proof}

To help us prove our main result, it is useful to have the following two lemmas.
The first says that to check if an inverse monoid is Boolean it is enough to prove that we have orthogonal joins.

\begin{lemma}\label{lem:king} Let $S \subseteq R$, where $R$ is a unital ring.
Suppose that $0,1 \in S$ and that $S$ is an inverse monoid under the induced multiplication from $R$.
In addition, we suppose that $S$ is closed under sums of pairs of orthogonal elements and,
if $e \in S$ is an idempotent, then $1 - e \in S$.
Then $S$ is a Boolean inverse monoid.
\end{lemma}
\begin{proof} It is immediate that $\mathsf{E}(S)$ satisfies Frink's axioms and so is a Boolean algebra.
Let $e$ and $f$ be orthogonal idempotents.
Then $e + f = e \vee f$; to see why
observe that $(e + f)e = e$ and $(e + f)f = f$.
Thus $e,f \leq e + f$.
On the other hand, suppose that $e,f \leq i$.
Then $(e + f)i = e + f$.
Thus $e + f \leq i$.
It follows that $e + f = e \vee f$.
Let $a,b \in S$ be arbitrary orthogonal elements.
They are, in particular, compatible.
We prove that $a + b = a \vee b$ and the result then follows by Lemma~\ref{lem:queen}.
Observe that
$$(a + b)\mathbf{d}(a) = a \mathbf{d}(a) + b\mathbf{d}(b) = a \mathbf{d}(a) + a \mathbf{d}(b) = a\mathbf{d}(a) = a$$
where we have used  Lemma~\ref{lem:compatibility-meets}.
Thus $a \leq a + b$.
Similarly $b \leq a + b$.
Suppose that $a,b \leq x$.
Then $a + b = x\mathbf{d}(a) + x\mathbf{d}(b) = x(\mathbf{d}(a) \vee \mathbf{d}(b))$.
Thus $a + b \leq x$.  
We have proved that $a + b = a \vee b$ if $a$ and $b$ are orthogonal.
\end{proof}

Our second result deals with orthogonal sums.

\begin{lemma}\label{lem:biden} Let $S \subseteq R$, where $R$ is a unital ring.
Suppose that $0,1 \in S$ and that $S$ is an inverse monoid under the induced multiplication from $R$.
Define $S^{\perp}$ to be the sums of all finite orthogonal sets of $S$.
Then $S^{\perp}$ is an inverse monoid and its set of idempotents is $\mathsf{E}(S)^{\perp}$.
\end{lemma}
\begin{proof} Let $a,b \in S^{\perp}$.
Then the fact that $ab \in S^{\perp}$ follows by Lemma~\ref{lem:gill}.
Observe that if $a = a_{1} + \ldots + a_{n}$, an orthogonal sum of elements of $S$,
then defining $a^{-1} = a_{1}^{-1}+ \ldots + a_{n}^{-1}$ we get that $a = aa^{-1}a$.
Thus $S^{\perp}$ is a regular semigroup
where, for example, $aa^{-1} =  a_{1}a_{1}^{-1} + \ldots + a_{n}a_{n}^{-1}$.
Suppose that $a = a_{1} + \ldots + a_{n}$ is an idempotent.
Then $a^{2} = a$ and so $a^{2}a_{1}^{-1} = a_{1}a_{1}^{-1}$.
It follows that $a_{1}^{2}a_{1}^{-1} = a_{1}a_{1}^{-1}$.
We deduce that $a_{1}$ is an idempotent.
In a similar way, we may deduce that each of $a_{1},\ldots,a_{n}$ is an idempotent
and so $a$ is a sum of orthogonal idempotents.
We have therefore provded that $\mathsf{E}(S^{\perp}) = \mathsf{E}(S)^{\perp}$.
It now follows that $S^{\perp}$ is an inverse monoid since it is regular and the idempotents commute.
\end{proof}

We can now prove the more general result which shows us that Boolean inverse monoids arise naturally.

\begin{theorem}\label{them:bundestag} 
Let $S \subseteq R$, where $R$ is a unital ring.
Suppose that $0,1 \in S$ and $S$ is an inverse monoid under the induced multiplication from $R$.
Then there is a subset $S \subseteq T \subseteq R$
such that $T$ is a Boolean inverse monoid.
\end{theorem}
\begin{proof} Put $E = \mathsf{E}(S)$.
Define $E'$ as in the proof of Lemma~\ref{lem:boolean-ring}.
It is easy to check that $E'$ is closed under conjugation by elements of $S$.
Put $S' = SE'$.
Then $S'$ is a regular monoid and the set of idempotents of $S'$ is precisely the set $E'$.
It follows by Proposition~\ref{prop:idempotents-commute}, 
that $S'$ is also an inverse monoid.
Define $T$ to be the sums of finite orthogonal subsets of $S'$.
Then $T$ is an inverse monoid whose set of idempotents is precisely the set $(\mathsf{E}(S)')^{\perp}$,
using the notation of the proof of Lemma~\ref{lem:boolean-ring}.
This is proved using Lemma~\ref{lem:biden}.
It now follows by Lemma~\ref{lem:king} that $T$ is a Boolean inverse monoid. 
\end{proof}

\section{The underlying groupoid}

The product we have defined on the symmetric inverse monoid $\mathcal{I}(X)$ is not the only one,
nor perhaps even the most obvious one, that we might define.
Given partial bijections $f$ and $g$, we could define $fg$ only when the domain of $f$ is equal to the range of $g$.
When we do this, we are regarding $f$ and $g$ as being functions rather than as partial functions.
With respect to this `restricted product', $\mathcal{I}(X)$ becomes a groupoid.
What we have done for the special case of the symmetric inverse monoids,
we can also do for arbitrary inverse semigroups,
but first we review the basics of groupoid theory we shall need.

Categories are usually regarded as categories of structures with their morphisms.
They can, however, also be regarded as algebraic structures in their own right,
no different from groups, rings and fields except that the binary operation
is only partially defined.
We shall define categories from this purely algebraic point of view.
Let $C$ be a set equipped with a partial binary operation 
which we shall denote by $\cdot$ or by concatenation.  
If $x,y \in C$ and the product $x \cdot y$ is defined we write 
$\exists x \cdot y$.  
An element $e \in C$ is called an {\em identity}
if 
$\exists e \cdot x$ implies $e \cdot x=x$, 
and $\exists x \cdot e$ implies $x \cdot e=x$. 
The set of identities of $C$ is denoted $C_{o}$,
where the subscript `o' stands for `object'.
The pair $(C, \cdot )$ is said to be a {\em category} 
if the following  axioms hold:  
\begin{description}

\item[{\rm (C1)}] $x \cdot (y \cdot z)$ exists 
if and only if $(x \cdot y) \cdot z$ exists, in which case they are equal. 

\item[{\rm (C2)}] $x \cdot (y \cdot z)$ exists if and only if 
$x \cdot y$ and $y \cdot z$ exist. 

\item[{\rm (C3)}] For each $x \in C$ there exist identities 
$e$ and $f$ such that $\exists x \cdot e$ and $\exists f \cdot x$. 

\end{description}
From axiom (C3), it follows that the identities $e$ and $f$ 
are uniquely determined by $x$. 
We write $e = {\bf d} (x)$ and $f = {\bf r} (x)$, where 
${\bf d}(x)$ is the {\em domain} identity and ${\bf r}(x)$ 
is the {\em range} identity.\footnote{As you will see, there is no contradiction with the notation we introduced earlier.}
Observe that $\exists x \cdot y$ if
and only if ${\bf d} (x) = {\bf r} (y)$. 
The elements of a category are called {\em arrows}.
We say that the arrow $x$ {\em starts at} $\mathbf{d}(x)$ and {\em ends at} $\mathbf{r}(x)$.
If $C$ is a category and $e$ and $f$ identities 
in $C$ then we put 
$$\mbox{\rm hom}(e,f) 
= \{x \in C \colon \: {\bf d}(x) = f \mbox{ and } {\bf r}(x) = e\},$$
the set of {\em arrows from $f$ to $e$}.
Subsets of $C$ of the form $\mbox{\rm hom}(e,f)$
are called {\em hom-sets}.\index{hom-set}
We also put $\mbox{\rm end}(e) = \mbox{\rm hom}(e,e)$,
the {\em local monoid at $e$}.
We define {\em subcategories} in the obvious way.
Viewed in this light, a catgeory is a monoid with many identities since
the categories with exactly one identity are precisely the monoids.

We shall not need arbitary categories but those that generalize groups.
A category $C$ is said to be a {\em groupoid} 
if
for each $x \in C$ there is an element $x^{-1}$ such that
$x^{-1}x = {\bf d}(x)$ and $xx^{-1} = {\bf r}(x)$.
The element $x^{-1}$ is unique with these properties. 
In the case of groupoids, the local monoids are, in fact, {\em local groups}.
It can happen that a groupoid consists entirely of its local units;
in this case, we say that the groupoid is a {\em union of groups}.
A groupoid with exactly one identity is a group.
Two elements $x$ and $y$ of a groupoid are said to be 
{\em connected}
if there is an arrow starting at ${\bf d}(x)$
and ending at ${\bf d}(y)$.
This defines an equivalence relation on the groupoid whose equivalence classes 
are called the {\em connected components}\footnote{The use of the word `connected' here is unfortunate. It has nothing to do with topology. The groupoids here are discrete.} of the groupoid.
A groupoid with one connected component is said to be {\em connected}. 
Every groupoid can be written as a disjoint union of connected groupoids.
We say that a groupoid is {\em principal} if for any identities $e$ and $f$
there is at most one arrow from $e$ to $f$.

\begin{lemma}\label{lem:principal-lg} Let $G$ be a groupoid.
Then $G$ is principal if and only if all its local groups are trivial.
\end{lemma}
\begin{proof} If a groupoid is principal, it is clear that the local groups are trivial.
Suppose, now, that the local groups are trivial.
We prove that the groupoid is principal.
Suppose that $x$ and $y$ are arrows that start at $e$ and end at $f$.
Then $y^{-1}x$ begins and ends at $e$.
This means that, under our assumption, $y^{-1}x = e$.
We deduce that $x = y$.
\end{proof}

We define {\em subgroupoids} in the obvious way.

\begin{example}\label{ex:principal}
{\em Let $X$ be a set.
The set $X \times X$ becomes a groupoid when we define
$\mathbf{d}(x,y) = (y,y)$, $\mathbf{r}(x,y) = (x,x)$ and $(x,y)^{-1} = (y,x)$;
define a partial product by $(x,y)(y,z) = (x,z)$.
Now, let $\sim$ be an equivalence relation on $X$.
Define $G$ to consist of those ordered pairs $(x,y)$ where $x \sim y$.
It is easy to check that $G$ is a subgroupoid of $X \times X$. 
It is, in fact, a principal groupoid.
If $G$ is an arbitrary principal groupoid, then it defines an equivalence relation
on the set $X = G_{o}$.
In fact, principal groupoids and equivalence relations are different ways of defining
the same thing.}
\end{example}

At this point, category theorists should look away since we shall convert a category into a semigroup.
If $C$ is a category as we have defined it above, 
then we can convert it into a semigroup with zero by adjoining a zero 
and defining all undefined products to be zero.
We denote the semigroup with zero that arises by $C^{0}$.

Now we can return to inverse semigroups.
Let $S$ be an arbitrary inverse semigroup.
Define the {\em restricted product}\footnote{Sometimes referred to as the {\em trace product}.} of two elements $s$ and $t$ in $S$
to be $s \cdot t = st$ if $s^{-1}s = tt^{-1}$ and undefined otherwise.
The following result simply tells us that what we expect to happen actually does happen.

\begin{proposition}\label{prop:restricted-product} Every inverse semigroup $S$ is 
a groupoid with respect to its restricted product.
\end{proposition}
\begin{proof}
We begin by showing that all idempotents
of $S$ are identities of $(S,\cdot)$.
Let $e \in S$ be an idempotent and suppose that $e \cdot x$ is defined.  
Then $e = xx^{-1}$ and $e \cdot x = ex$.  
But $ex = (xx^{-1})x = x$.  
Similarly, if $x \cdot e$ is defined then it is equal to $x$.  
We now check that the axioms (C1), (C2) and (C3) hold.

Axiom (C1) holds: suppose that $x \cdot (y \cdot z)$ is defined.  
Then 
$$x^{-1}x = (y \cdot z)(y \cdot z)^{-1} 
\mbox{ and } y^{-1}y = zz^{-1}.$$  
But 
$$(y \cdot z)(y \cdot z)^{-1} = yzz^{-1}y^{-1} = yy^{-1}.$$
Hence $x^{-1}x = yy^{-1}$, and so
$x \cdot y$ is defined.
Also $(xy)^{-1}(xy) = y^{-1}y = zz^{-1}$.
Thus $(x\cdot y) \cdot z$ is defined.
It is clear that $x \cdot (y \cdot z)$ is equal to
$(x \cdot y) \cdot z$. 
A similar argument shows that if $(x \cdot y) \cdot z$ exists 
then $x \cdot (y \cdot z)$ 
exists and they are equal. 

Axiom (C2) holds: suppose that $x \cdot y$ and $y \cdot z$ are defined.  
We show that $x \cdot (y \cdot z)$ is defined.  
We have that $x^{-1}x = yy^{-1}$ and $y^{-1}y = zz^{-1}$. 
Now 
$$(yz)(yz)^{-1} = y(zz^{-1})y^{-1} = y(y^{-1}y)y^{-1} = yy^{-1} = x^{-1}x.$$
Thus $x \cdot (y \cdot z)$ is defined.
The proof of the converse is straightforward.

Axiom (C3) holds: for each element $x$ we have that 
$x \cdot (x^{-1}x)$ is defined,
and we have seen that idempotents of $S$ are identities.  
Thus we put ${\bf d}(x) = x^{-1}x$.  
Similarly, we put $xx^{-1} = {\bf r}(x)$. 
It is now clear that $(S,\cdot)$ is a category.
The fact that it is a groupoid is immediate.
\end{proof}

We call $(S,\cdot)$ the {\em (underlying) groupoid} of $S$.
We can use the underlying groupoid to reveal something of the structure of inverse semigroups.

\begin{example}{\em We can interpret Lemma~\ref{lem:union-of-groups} in terms of the structure of the underlying groupoid. 
An inverse semigroup is a Clifford semigroup if and only if
its underlying groupoid is a union of groups.}
\end{example}

In the light of Proposition~\ref{prop:restricted-product},
it is now natural to picture an element $a$ of an inverse semigroup as follows:
$$\mathbf{d}(a) \stackrel{a}{\longrightarrow} \mathbf{r}(a).$$
where we now call $\mathbf{d}(a)$ the {\em domain idempotent} of $a$ and
$\mathbf{r}(a)$ is the {\em range idempotent of $a$}.
The underlying groupoid structure more or less does away with the need to deal directly with Green's relations.
If $e$ and $f$ are idempotents we write $e\, \mathscr{D}\, f$ to mean 
that there exists an element $a$ such that $\mathbf{d}(a) = e$ and $\mathbf{r}(a) = f$.
We write $a\, \mathscr{D}\, b$ to mean that  $\mathbf{d}(a)\, \mathscr{D}\, \mathbf{d}(b)$.
The relation $\mathscr{D}$ really is Green's relation  $\mathscr{D}$;
by the same token $a \, \mathscr{L} \, b$ if and only if $\mathbf{d}(a) = \mathbf{d}(b)$  
and $a \, \mathscr{R} \, b$ if and only if $\mathbf{r}(a) = \mathbf{r}(b)$.
We define $a \, \mathscr{H} \, b$ if and only if  $a \, \mathscr{L} \, b$ and $a \, \mathscr{R} \, b$;
in other words, $a$ and $b$ belong to the same {\em hom-set}.
An inverse semigroup is said to be {\em bisimple} if its underlying groupoid consists of one connected component.
An inverse semigroup with zero is said to be {\em $0$-bisimple} if its underlying groupoid consists of two connected components.

In passing from an inverse semigroup to its underlying groupoid, we do lose some information.
As the following result shows, the information that is lost is encoded by the natural partial order.

\begin{lemma}\label{lem:restricted-product} Let $S$ be an inverse semigroup.
Then for any $s,t \in S$ there exist elements $s' \leq s$ and $t' \leq t$ such that
$st = s' \cdot t'$ where the product on the right is the restricted product.
\end{lemma}
\begin{proof} Put $e = \dom (s) \ran (t)$ and define $s' = se$ and $t' = et$.
Observe that $\dom (s') = e$ and $\ran (t') = e$ and that $st = s't'$.
\end{proof}

It is possible to formalize the idea of a groupoid equipped with a partial order
in such a way that the original semigroup product can be recaptured.
This is the approach that Ehresmann took to studying inverse semigroups.
See \cite[Chapter]{Lawson} for exactly how this is done.\footnote{This groupoid, which is discrete, is quite different from the one that Paterson constructs \cite{Paterson}.}

We shall now deal with the missing $\mathscr{J}$-relation.
First, we need a slight extension of Lemma~\ref{lem:restricted-product}.

\begin{lemma}\label{lem:ext-restricted-product} Let $S$ be an inverse semigroup.
Then $abc = a' \cdot b' \cdot c'$, which is a restricted product where $a' \leq a$, $b' \leq b$ and $c' \leq c$. 
\end{lemma}
\begin{proof} Because of associativity, it doesn't matter how we bracket.
We write $abc = (ab)c$.
Put $d = ab$.
Then $dc = d' \cdot c'$.
But $d' = ab\mathbf{r}(c)$.
Thus $(a(b\mathbf{r}(c))) \cdot c'$.
Now write $a(b\mathbf{r}(c))  = a' \cdot b'$
where $b' \leq b\mathbf{r}(c) \leq b$.
We have therefore written $abc = a' \cdot b' \cdot c'$ using the fact that the multiplication in a groupoid is associative.
\end{proof}
We can now describe the $\mathscr{J}$-relation on inverse semigroups.

\begin{lemma}\label{lem:jay} Let $S$ be an inverse semigroup.
Then $SaS \subseteq SbS$ if and only if $a \, \mathscr{D} \, b' \leq b$ for some element $b' \in S$.
\end{lemma}
\begin{proof} Suppose, first, that  $SaS \subseteq SbS$.
Then $a = xby$ for some $x,y \in S$.
By Lemma~\ref{lem:ext-restricted-product}, 
we may write $a = x' \cdot b' \cdot y'$.
You can check that $a \, \mathscr{D} \, b'$.
But $b' \leq b$.
We have therefore proved one direction.
To prove the converse,
suppose that $a \, \mathscr{D} \, b' \leq b$ for some element $b' \in S$.
Let $\stackrel{x}{\mathbf{d}(a) \longrightarrow \mathbf{d}(b')}$.
Put $y = b' \cdot x \cdot a^{-1}$.
Then $a = y^{-1} \cdot b' \cdot x$.
Similarly, $b' = y \cdot a \cdot x^{-1}$
It follows that $SaS = Sb'S$.
Clearly, $Sb'S \subseteq SbS$.
We have therefore shown that $SaS \subseteq SbS$.
\end{proof}

We can use the theory we have developed to generalize Lemma~\ref{lem:unique-idempotents}.

\begin{lemma}\label{lem:groupoids-as-inverse} 
Let $S$ be an inverse semigroup with zero.
Then $S$ is isomorphic to a groupoid with a zero adjoined if and only if
the natural partial order is equality on the set $S \setminus \{0\}$.
\end{lemma}   
\begin{proof} We prove one direction only.
Let $S$ be an inverse semigroup with zero such that the natural partial order is equality on the set $S \setminus \{0\}$.
Let $G$ be the underlying groupoid of $S$ with the component $\{0\}$ removed.
There is an obvious bijection between $S$ and $G^{0}$.
We prove that this is a homomorphism.
Let $a,b \in S$ be any non-zero elements.
Then $ab = a' \cdot b'$ where $a' \leq a$ and $b' \leq b$ by Lemma~\ref{lem:restricted-product}.
As a result of our assumption on the natural partial order,
either $a' = a$ and $b' = b$, in which case the product is a groupoid product,
or at least one of $a'$ and $b'$ is equal to zero --- in which case $ab = 0$.
Thus a product in $S$ is either equal to zero or a restricted product but not both.
\end{proof}

In addition to the underlying groupoid, we may sometimes be able to associate another, smaller, groupoid to an inverse semigroup with zero.
Let $S$ be an inverse semigroup with zero.
An element $s \in S$ is said to be an {\em atom} if $t \leq s$ implies that $t = 0$ or $t = s$.

\begin{lemma}\label{lem:atomic-groupoid} 
Let $S$ be an inverse semigroup.
\begin{enumerate}
\item If $x$ is an atom then $x^{-1}$ is an atom. 

\item If $x$ is an atom then both $\mathbf{d}(x)$ and $\mathbf{r}(x)$ are atoms

\item Suppose that  $\mathbf{d}(x)$ is an atom then $x$ is an atom;
similarly, if $\mathbf{r}(x)$ is an atom then $x$ is an atom.

\item If $x$ and $y$ are atoms and the restricted product $x \cdot y$ is defined then $x \cdot y$ is an atom.

\item  If $x$ and $y$ are distinct compatible atoms then $x \perp y$.

\end{enumerate}
\end{lemma}
\begin{proof} (1) Immediate from the properties of the inverse.

(2) We prove that $\mathbf{d}(x)$ is an atom;
the proof that $\mathbf{r}(x)$ is an atom is similar.
Suppose that $e \leq \mathbf{d}(x)$.
Then $xe \leq x$.
It follows that either $xe = 0$ or $xe = x$.
Suppose, first, that  $xe = 0$.
Then $e\mathbf{d}(x) = 0$ and so $e = 0$.
Alternatively, if $xe = x$ then $e\mathbf{d}(x) = \mathbf{d}(x)$ and so $e = \mathbf{d}(x)$.
This shows that $\mathbf{d}(x)$ is an atom.

(3) We prove that if $\mathbf{d}(x)$ is an atom then $x$ is an atom;
the proof of the other statement is analogous.
Suppose that $y \leq x$.
Then $\mathbf{d}(y) \leq \mathbf{d}(x)$.
Since $\mathbf{d}(x)$ is an atom either $\mathbf{d}(y) = 0$ or $\mathbf{d}(y) = \mathbf{d}(x)$.
It follows that either $y = 0$ or $y = x$.
We have therefore proved that $x$ is an atom.  

(4) We are given that $x$ and $y$ are atoms and that $x \cdot y$ exists.
Observe first that $x \cdot y \neq 0$.
Let $z \leq xy$.
Then $z = x (y\mathbf{d}(z))$.
Now, $y\mathbf{d}(z) = 0$ or $y\mathbf{d}(z) = y$, since $y$ is an atom.
If the former then $z = 0$ and if the latter then $z = xy$.
We have therefore proved that $x \cdot y$ is an atom.

(5) Let $x$ be an atom.
Then by part (2) above it follows that $\mathbf{d}(x)$ is an atom.
Similarly, $\mathbf{d}(y)$ is an atom.
If the product $\mathbf{d}(x)\mathbf{d}(y)$ is non-zero
then in fact $\mathbf{d}(x) = \mathbf{d}(y)$.
But $x \sim y$ and so $x = y$ by Lemma~\ref{lem:karen}, which contradicts our assumption that $x$ and $y$ are distinct.
It follows that $\mathbf{d}(x) \perp \mathbf{d}(y)$.
A similar argument shows that $\mathbf{r}(x) \perp \mathbf{r}(y)$
from which it follows that $x \perp y$.
\end{proof}

It follows by Lemma~\ref{lem:atomic-groupoid},
that the set of atoms of $S$, if non-empty, forms a groupoid, which we shall call the {\em atomic groupoid} of $S$ 
and denote by $\mathsf{A}(S)$. 

\begin{example}
{\em The finite symmetric inverse monoid $\mathcal{I}(X)$ has an interesting atomic groupoid.
It consists of those partial bijections the domains of which contain exactly one element of $X$.
This groupoid is isomorphic to the groupoid $X \times X$ defined in Example~\ref{ex:principal}.}
\end{example}

We shall describe  all finite Boolean inverse monoids in terms of groupoids.
The reader will recall that the finite Boolean algebras are isomorphic to the powerset Boolean algebras
defined on the finite set of atoms.
We shall replace finite sets by finite groupoids.
We need some definitions first.
If $A$ and $B$ are subsets of a category $C$ then $AB$ is the set of all products
$ab$ where $a \in A$, $b \in B$ and $\mathbf{d}(a) = \mathbf{r}(b)$.
If $A$ is a subset of a groupoid then $A^{-1}$ is the set of all $a^{-1}$ where $a \in A$.
We first show how to construct finite Boolean inverse monoids from finite groupoids.
If $G$ is a groupoid then a subset $A \subseteq G$ is said to be a {\em local bisection}
if both $AA^{-1}$ and $A^{-1}A$ consist entirely of identities.
You can check that a subset $A \subseteq G$ is a local bisection if $a,b \in A$ and $\mathbf{d}(a) = \mathbf{d}(b)$ 
(respectively, $\mathbf{r}(a) = \mathbf{r}(b)$) 
then $a = b$, 

\begin{proposition}\label{prop:groupoids} 
Let $G$ be a finite groupoid.
Then $\mathsf{K}(G)$, the set of all local bisections of $G$ under subset multiplication,
is a finite Boolean inverse monoid, the set of atoms of which forms a groupoid isomorphic to $G$.
\end{proposition} 
\begin{proof} You can check that the product of two local bisections is a local bisection.
If $A$ is a local bisection, then so is $A^{-1} =  \{a^{-1} \colon a \in A\}$ 
and $A = AA^{-1}A$.
The idempotents are just the subsets of $G_{o}$ and the
product of two idempotents is just the intersection of these two sets.
It follows that $\mathsf{K}(G)$ is an inverse semigroup since it is a regular semigroup with commuting idempotents.
It is a monoid with identity $G_{o}$ and has a zero $\varnothing$.
It has a Boolean algebra of idempotents.
Observe that $A \leq B$ if and only if $A \subseteq B$.
You can check that $A \sim B$ if and only if $A \cup B$ is a local bisection.
It is now easy to check that $\mathsf{K}(G)$ is a Boolean inverse monoid.
The atoms are the singleton sets $\{g\}$ and form a groupoid
isomorphic to $G$.
\end{proof}

We shall now go in the opposite direction.
Our first result is an immediate consequence of finiteness.

\begin{lemma}\label{lem:atoms} Let $S$ be a finite Boolean inverse monoid.
Then each non-zero element is above an atom.
\end{lemma}

We now connect elements with the atoms beneath them.
Let $a \in S$.
Define $\theta (a) = a^{\downarrow} \cap \mathsf{A}(S)$.
Observe that $\theta (0) = \varnothing$. 

\begin{lemma}\label{lem:two} 
Let $S$ be a Boolean inverse semigroup.
For each $a \in S$, the set $\theta (a)$ is a local bisection of the groupoid $\mathsf{A}(S)$.
\end{lemma} 
\begin{proof} If $a = 0$ then $\theta (a) = \varnothing$.
If $a \neq 0$ then it is above at least one atom by Lemma~\ref{lem:atoms} and so is non-empty.
Let $x,y \in \theta (a)$ such that $\mathbf{d}(x) = \mathbf{d}(y)$.
Then $x = y$.
Dually, if $\mathbf{r}(x) = \mathbf{r}(y)$ then $x = y$.
\end{proof}

If $S$ is a finite Boolean inverse monoid, then by Lemma~\ref{lem:atoms} we may define a function $\theta \colon S \rightarrow \mathsf{K}(\mathsf{A}(S))$
by $\theta (0) = \varnothing$ and  $\theta (a) = a^{\downarrow} \cap \mathsf{A}(S)$.
It is quite rare that we can say anything about the structure of finite semigroups belonging to some class.
Thus the following result is a pleasant surprise.

\begin{theorem}[The structure of finite Boolean inverse monoids]\label{them:main-finite} Let $S$ be a finite Boolean inverse monoid.
Then $S$ is isomorphic to the Boolean inverse monoid $\mathsf{K}(\mathsf{A}(S))$.
\end{theorem}
\begin{proof} It remains to show that $\theta$ (as defined above) is an isomorphism of semigroups.
First, $\theta$ is a homomorphism.
Let $x$ be an atom such that $x \leq ab$.
Then $x = a(b\mathbf{d}(x))$.
Thus by  Lemma~\ref{lem:restricted-product}, we may write
$x = a' \cdot b'$ where $a' \leq a$ and $b' \leq b$.
It is easy to check that $a'$ and $b'$ are themselves atoms.
We have therefore proved that $\theta (ab) \subseteq \theta (a)\theta (b)$.
Conversely, let $x \in \theta (a)$ and $y \in \theta (b)$ such that the restricted product $x \cdot y$ is defined.
Then $x \cdot y = xy \leq ab$.
But the restricted product of atoms is an atom and so we have proved the first claim.

It remains to prove that $\theta$ is a bijection.
We show first that $a = \bigvee \theta (a)$.
Put $b = \bigvee \theta (a)$.
Then, clearly, $b \leq a$.
Suppose that $b \neq a$.
It is here that we use the Boolean structure.
Then $a \setminus b \neq 0$.
It follows by Lemma~\ref{lem:atoms} that $a \setminus b$ is above an atom $x$.
But then $x \leq a$ and so $x \leq b$ (by definition).
Thus $x \leq b, a \setminus b$ which implies that $x = 0$.
But atoms are non-zero.
It follows that $a = \bigvee \theta (a)$
and so $\theta$ is an injection.

Now let $A \in \mathsf{K}(\mathsf{A}(S))$.
We don't lose any generality by assuming that it is non-empty.
Then $A$ is a set of compatible elements.
Put $a = \bigvee A$.
Clearly, $A \subseteq \theta (a)$.
Let $x$ be an atom such that $x \leq a$.
Then $x = \bigvee_{a \in A} (x \wedge a)$ by part (2) of Lemma~\ref{lem:domains}.
Remembering that both $x$ and $a$ are atoms,
it follows that $x = a$ for some $a \in A$.
We have therefore proved that $\theta (a) = A$.
\end{proof}

We now have a complete description of the finite Boolean inverse monoids:
as a result of Proposition~\ref{prop:groupoids} and Theorem~\ref{them:main-finite},
they are precisely the inverse monoids of the form $\mathsf{K}(G)$ where $G$ is a finite groupoid.
We can apply the theory we have developed to the representation theory of {\em arbitrary} finite inverse monoids,
although a more elementary account can be found in \cite[Chapter 9]{Steinberg2016}.
We prove first that every finite inverse monoid $S$ can be embedded into the Boolean
inverse monoid constructed from its underlying groupoid.

\begin{lemma}\label{lem:one} Let $S$ be a finite inverse monoid with underlying groupoid $G$.
\begin{enumerate}
\item Let $a \in S$. Then $a^{\downarrow}$ is a local bisection of $G$.
\item We have that $(ab)^{\downarrow} = a^{\downarrow}b^{\downarrow}$.
\end{enumerate}
\end{lemma}
\begin{proof} (1) Let $x,y \leq a$ such that $\mathbf{d}(x) = \mathbf{d}(y)$.
Then it is immediate that $x = y$.
Similar reasoning shows that if  $x,y \leq a$ such that $\mathbf{r}(x) = \mathbf{r}(y)$ then $x = y$.
We have therefore shown that the set $a^{\downarrow}$ is a local bisection of $G$.

(2) Suppose that $x \leq a$ and $y \leq b$.
Then $xy \leq ab$.
On the other hand, if $c \leq ab$ then we can write
$c = (\mathbf{r}(c)a)(b\mathbf{d}(c))$ and so $x \in a^{\downarrow}b^{\downarrow}$.
\end{proof}

Let $S$ be an arbitrary inverse monoid with underlying groupoid $G$.
An element $a$ with the property that $a^{\downarrow} = \{a\}$, a singleton set, will be said to be {\em at the bottom}.
Define $\beta \colon S \rightarrow \mathsf{K}(G)$ by
$\beta (a) = a^{\downarrow}$. 
Then by Lemma~\ref{lem:one}, $\beta$ is an injective homomorphism of inverse semigroups.
The elements in $1^{\downarrow}$ are precisely the idempotents,
which are the identities of $G$; the set of identities of $G$ is the monoid identity of $\mathsf{K}(G)$.
Thus the homomorphism is a monoid homomorphism.
The above result, on its own, doesn't take us very far because we already know that
every finite inverse monoid can be embedded in a finite Boolean inverse monoid.
The key point of this embedding is that the map $\beta \colon S \rightarrow \mathsf{K}(G)$ has a special property described by the following proposition.

\begin{proposition}\label{prop:entourage} Let $S$ be a finite inverse monoid with underlying groupoid $G$
and let $\alpha \colon S \rightarrow T$ be any monoid homomorphism to a Boolean inverse monoid $T$.
Then there is a unique morphism of Boolean inverse monoids $\gamma \colon \mathsf{K}(G) \rightarrow T$
such that $\gamma \beta = \alpha$.
\end{proposition} 
\begin{proof} Our first step is to define the function  $\gamma \colon \mathsf{K}(G) \rightarrow T$.
Define $\gamma (\varnothing) = 0$.
If $a$ is any element of $S$ then $\{a\}$ is an element of $\mathsf{K}(G)$.
Let the set of elements strictly below $a$ in $S$ be $\{a_{1}, \ldots, a_{m}\}$.
This set could well be empty if $a$ is at the bottom --- this will not cause us any problems.
Since $a_{1},\ldots, a_{m} \leq a$,
these elements are pairwise compatible.
It follows that $\alpha (a_{1}),\ldots, \alpha (a_{m})$ is a compatible subset of $T$.
Thus the join $\alpha (a_{1}) \vee \ldots \vee \alpha (a_{m})$ exists;
if the set of elements strictly below $a$ is empty then this join is just $0$.
It is clearly less than $\alpha (a)$ and so we may form the element
$\alpha (a) \setminus (\alpha (a_{1}) \vee \ldots \vee \alpha (a_{m}))$.
On the basis of the above, define 
$$\gamma (\{a\}) = \alpha (a) \setminus (\alpha (a_{1}) \vee \ldots \vee \alpha (a_{m})).$$
If $A$ is a local bisection which is neither empty nor a singleton set,
define 
$$\gamma (A) = \bigvee_{a \in A} \gamma (\{a\}).$$
This makes sense, for if
$a,b \in A$, then $\{a\}, \{b\} \subseteq A$ and so $\{a\} \sim \{b\}$ in $\mathsf{K}(G)$.
This completes the definition of $\gamma$.

We show that $\gamma$ preserves binary joins.
If $A \sim B$ in $\mathsf{K}(G)$ then $A \vee B = A \cup B$.
From the definition of $\gamma$, it is clear that $\gamma$ preserves binary joins.
It maps the empty set to the zero by definition.

We prove that $\gamma (a^{\downarrow}) = \alpha (a)$.
This proves that we have a monoid homomorphism
and that $\gamma \beta = \alpha$.
Define the {\em height} of $a$ in $S$ to be the length of a chain of maximum length from $a$.
Those elements with height zero are precisely those which are at the bottom.
If $a$ has height zero then 
$$\gamma (a^{\downarrow}) = \alpha (a).$$  
We assume that we have proved that $\gamma (b^{\downarrow}) = \alpha (b)$ for all elements $b$ of height at most $n$.
Let $a$ be an element of height $n + 1$.
Let $b_{1},\ldots, b_{m}$ be all the elements immediately below $a$.
Then, by the induction hypothesis, we have that
$$\gamma (b_{i}^{\downarrow}) = \alpha (b_{i}).$$
Let the elements strictly less than $a$ be $a_{1},\ldots, a_{m}$.
These include the elements $b_{1},\ldots, b_{m}$, for example, so in general
each element $a_{i}$ is beneath one of the $b_{j}$.
It follows that we can write 
$$\alpha (a) \setminus (\alpha (a_{1}) \vee \ldots \vee \alpha (a_{m}))$$
as 
$$\alpha (a) \setminus (\alpha (b_{1}e_{1}) \vee \ldots \vee \alpha (b_{n}e_{n}))$$
where the idempotents $\alpha (e_{j})$ gather together by means of a join 
all the idempotents that arise from showing that $a_{i} \leq b_{j}$ for various $i$.
By definition
$$\gamma (a^{\downarrow}) = \alpha (a) \setminus (\alpha (a_{1}) \vee \ldots \vee \alpha (a_{m})) \vee \gamma (\{a_{1}\}) \vee \ldots \vee \gamma (\{a_{m}\})$$
but we can write 
$$\gamma (a^{\downarrow}) = \gamma (\{a\}) \vee \gamma (\{a_{1},\ldots, a_{m}\}^{\downarrow}).$$
But 
$$\{a_{1},\ldots, a_{m}\}^{\downarrow} = b_{1}^{\downarrow} \cup \ldots \cup b_{n}^{\downarrow}$$
and so 
$$\gamma (\{a_{1},\ldots, a_{m}\}^{\downarrow}) = \gamma (b_{1}^{\downarrow}) \vee \ldots \vee \gamma (b_{n}^{\downarrow}).$$
Thus 
$$\gamma (\{a_{1},\ldots, a_{m}\}^{\downarrow}) = \alpha (b_{1}) \vee \ldots \vee \alpha (b_{n})$$
using the induction hypothesis.
Thus
$$\gamma (a^{\downarrow}) = \alpha (a) \setminus (\alpha (a_{1}) \vee \ldots \vee \alpha (a_{m})) \vee \alpha (b_{1}) \vee \ldots \vee \alpha (b_{n}.)$$
By our argument above
$$\gamma (a^{\downarrow}) = \alpha (a) \setminus (\alpha (b_{1}e_{1}) \vee \ldots \vee \alpha (b_{n}e_{n})) \vee \alpha (b_{1}) \vee \ldots \vee \alpha (b_{n}).$$
This is equal to 
$$\alpha (a) \setminus (\alpha (b_{1}e_{1}) \vee \ldots \vee \alpha (b_{n}e_{n})) 
\vee
( \alpha (b_{1}e_{1}) \vee \ldots \vee \alpha (b_{n}e_{n}))                        )
\vee \alpha (b_{1}) \vee \ldots \vee \alpha (b_{n})$$
which is just 
$$\alpha (a) \vee  \alpha (b_{1}) \vee \ldots \vee \alpha (b_{n})$$
which is equal to $\alpha (a)$.

We now prove uniqueness.
Suppose that $\gamma' \colon \mathsf{K}(S) \rightarrow T$ is a morphism of Boolean inverse monoids
such that $\gamma' \beta = \alpha$.
We show that $\gamma' = \gamma$.
It is immediate from our assumption on $\gamma'$ that for all elements of height zero
we have that $\gamma (a) = \gamma' (a)$.
So, let $a \in S$ be any  element which does not have height zero. 
Let $\{a_{1},\ldots, a_{m}\}$ be the set of all elements strictly less than $a$.
Then 
$$\{a_{1},\ldots, a_{m}\} = a_{1}^{\downarrow} \cup \ldots \cup a_{m}^{\downarrow}.$$
It follows that 
$$\gamma' (\{a_{1},\ldots, a_{m}\}) = \gamma' (a_{1}^{\downarrow}) \vee \ldots \vee \gamma' (a_{m}^{\downarrow}).$$
But this is just equal to 
$$\gamma' (\beta (a_{1})) \vee \ldots \vee \gamma' (\beta (a_{m}))$$
which is equal to 
$$\gamma (\beta (a_{1})) \cup \ldots \cup \gamma (\beta (a_{m}))$$ 
by assumption.
We have therefore proved that 
$$\gamma' (\{a_{1},\ldots, a_{m}\}) = \gamma (\{a_{1},\ldots, a_{m}\}).$$
Observe that $\{a\} =  a^{\downarrow}\setminus  \{a_{1},\ldots, a_{m}\}$.
Thus 
$$\gamma' (\{a\}) = \gamma' (a^{\downarrow}) \setminus  \gamma'(\{a_{1},\ldots, a_{m}\}).$$
But 
$$\gamma' (a^{\downarrow}) = \gamma (a^{\downarrow})$$ 
by assumption,
and 
$$\gamma'(\{a_{1},\ldots, a_{m}\}) = \gamma (\{a_{1},\ldots, a_{m}\})$$
by what we proved above.
It follows that $\gamma' (\{a\}) = \gamma (\{a\})$ for all elements $a \in S$.
Now, let $A$ be any non-empty local bisection.
Then, since $\gamma'$ is a morphism of Boolean inverse monoids, 
we have that $\gamma' (A) = \bigvee_{a \in A} \gamma' (\{a\})$.
It follows that $\gamma' = \gamma$.
\end{proof}

We now apply our results to the study of the representation theory of finite inverse monoids.
The starting point is to say what we mean by the representation theory of groupoids.
Let $G$ be a finite groupoid.
Then we get a semigroup with zero $G^{0}$ by adjoining a zero;
in fact, this is a special kind of inverse semigroup by Lemma~\ref{lem:groupoids-as-inverse}.
We can therefore consider homomorphisms $\theta \colon G^{0} \rightarrow R$
to the multiplicative monoid of the unital ring $R$.
Each identity $e \in G$ gives rise to an idempotent $\theta (e)$ in the ring $R$.
If $e$ and $f$ are distinct identities of the groupoid $G$ then $\theta (e)$ and $\theta (f)$ are orthogonal.
By assumption, the groupoid $G$ is finite.
Put $f = \sum_{e \in G_{o}} \theta (e)$.
Then $f$ is an idempotent in $R$.
Consider the ring $fRf = \{a \in R \colon faf = a\}$.
This has identity $f$.
Let $g \in G$ be an arbitary element of $G$.
Let $e,e'$ be the identities such that $a = e'ae$.
Then $\theta (a) = \theta (e') \theta (a) \theta (e)$.
But for any identity $e$, we have that $\theta (e) = f\theta (e)f$.
It follows that $\theta (a) \in fRf$.
We therefore define a {\em representation} of a finite groupoid $G$ in a ring $R$
to be a semigroup with zero homomorphism $\theta \colon G^{0} \rightarrow R$ 
such that $1 = \sum_{e \in G_{o}} \theta (e)$.

We shall also need the following definition.
Let $S$ be a Boolean inverse monoid.
A monoid homorphism $\phi \colon S \rightarrow R$ to the multiplicative monoid
of the ring $R$ which maps zero to zero is said to be {\em additive} 
if $\phi (a \vee b) = \phi (a) + \phi (b)$, whenever $a \perp b$.
We refer the reader to part (2) of Lemma~\ref{lem:queen} for the rationale for this definition.

The following theorem was proved in a different way in \cite[Theorem 9.3]{Steinberg2016}.

\begin{theorem}[Representation theory of finite inverse monoids]\label{them: representation-theory-of-finite-inverse-semigroups}
Let $S$ be a finite inverse monoid with underlying groupoid $G$
and let $R$ be a unital ring.
Then there is a bijective correspondence between the set of representations of $S$ in $R$ 
and the set of representations of the finite groupoid $G$ in $R$.
\end{theorem}
\begin{proof} Let $\theta \colon S \rightarrow R$ be a monoid homomorphism to the multiplicative monoid of the ring $R$.
The image $\theta (S)$ is an inverse submonoid of the multiplicative monoid of the ring $R$ by Lemma~\ref{lem:homomorphisms}
and so $\theta (S)^{0}$, the inverse monoid $\theta (S)$ with the zero of the ring $R$ adjoined, is an inverse monoid with zero.
Thus by Theorem~\ref{them:bundestag}, there is a Boolean inverse monoid $T$
such that $\theta (S)^{0} \subseteq T \subseteq R$.
By Proposition~\ref{prop:entourage}, 
there is therefore a unique morphism of Boolean inverse monoids $\phi \colon \mathsf{K}(G) \rightarrow T$
such that $\phi \beta = \theta$.
We want to regard $\phi$ as a map from $\mathsf{K}(G)$ to the ring $R$.
This is an additive homomorphism.
We have proved that every homomorphism $\theta \colon S \rightarrow R$
gives rise to an additive homomorphism $\phi \colon \mathsf{K}(G) \rightarrow R$.
On the other hand, given an additive homorphism $\phi \colon \mathsf{K}(G) \rightarrow R$,
we can construct a monoid homomorphism $\phi \beta \colon S \rightarrow R$. 
This leads to a bijective correspondence between representations of $S$ in $R$ and additive homomorphisms of $\mathsf{K}(G)$ in $R$.

We described all finite Boolean inverse monoids in Theorem~\ref{them:main-finite}.
In what follows, therefore, we may assume that $S$ is a finite Boolean inverse monoid with atomic groupoid $G$.
Let $\theta \colon S \rightarrow R$ be an additive homomorphism of $S$.
Then, by restriction, we get a semigroup homomorphism $\theta'$ from $G^{0}$ to the ring $R$.
Since $S$ is a finite Boolean inverse monoid, the identity of $S$ is an orthogonal join of the atomic idempotents.
It follows that the sum of the idempotents $\theta' (e)$, where $e \in G_{o}$,
is equal to the identity of $R$.
Thus, we have defined a representation of the groupoid $G$.
We now go in the opposite direction.
Let $S$ be a finite Boolean inverse monoid with atomic groupoid $G$
and suppose that there a representation $\theta' \colon G^{0} \rightarrow R$.
We shall now define a homomorphism $\theta$ of $S$ that extends $\theta'$.
Let $a \in S$.
Define $\theta \colon S \rightarrow R$ by
$\theta (a) = \theta' (a_{1}) + \ldots + \theta' (a_{m})$
where $a_{1}, \ldots, a_{m} \leq a$
are all the atoms below $a$; we can assume this is an orthogonal set by part (5) of Lemma~\ref{lem:atomic-groupoid}
This is an additive homomorphism of $S$. 
We therefore have a bijective correspondence between the additive homorphisms of a Boolean inverse monoid
and the representations of its atomic groupoid.

If we put our two results together, then we have established a bijective correspondence between representations
of an inverse monoid $S$ and the representations of its underlying groupoid $G$.
\end{proof}

\section{Fundamental inverse semigroups}

In Section 3, we introduced the Clifford semigroups.
These are the inverse semigroups in which every element is central.
Living inside every inverse semigroup is a Clifford semigroup.
For every inverse semigroup $S$, define $\mathsf{Z}(\mathsf{E}(S))$, the {\em centralizer of the idempotents},
to be set of all elements of $S$ which commute with every idempotent.
Then $\mathsf{Z}(\mathsf{E}(S))$ is a wide inverse subsemigroup of $S$ which is Clifford.
If $\mathsf{Z}(\mathsf{E}(S)) = \mathsf{E}(S)$ we say the inverse semigroup is {\em fundamental}.
Fundamental inverse semigroups are important.
In this section, we shall study them in more detail.

To do this, we shall need the following.
Define the relation $\mu$ on an arbitrary inverse semigroup by
$$(s,t) \in \mu \Leftrightarrow (\forall e \in \mathsf{E}(S))(ses^{-1} = tet^{-1}).$$
It is routine to check that this is a congruence.

\begin{lemma}\label{lem:idpt-sep-mu} Let $S$ be an inverse semigroup.
\begin{enumerate}
\item If $(s,t) \in \mu$ then $\mathbf{r}(s) = \mathbf{r}(t)$ and $\mathbf{d}(s) = \mathbf{d}(t)$.
It follows that $\mu \subseteq \mathscr{H}$.
\item In the definition, we may restrict to those idempotents in $\mathbf{d}(s)^{\downarrow}$ 
\item If $(e,f) \in \mu$, where $e$ and $f$ are idempotents, then $e = f$.
\end{enumerate}
\end{lemma}
\begin{proof}
(1) If $(s,t) \in \mu$ then $\mathbf{r}(s) \leq \mathbf{r}(t)$.
Symmetry now delivers the answer.
To prove the second claim, observe that if $(s,t) \in \mu$ then $(s^{-1},t^{-1}) \in \mu$.
We can now use the first claim to prove the second claim.
It follows that $\mu \subseteq \mathscr{H}$.

(2) We have used all idempotents in the definition of $\mu$ but we may restrict 
to those idempotents in $\mathbf{d}(s)^{\downarrow}$ simply by mulitplying by $\mathbf{d}(s)$. 

(3) Immediate.
\end{proof}

By part (3) of Lemma~\ref{lem:idpt-sep-mu}, 
it follows that $\mu$ is an {\em idempotent-separating} congruence.
In fact, we have the following.

\begin{lemma}\label{lem:largest} 
$\mu$ is the largest idempotent-separating congruence on $S$.
\end{lemma}
\begin{proof}
Let $\rho$ be any idempotent-separating congruence on $S$ and let $(s,t) \in \rho$.
Let $e$ be any idempotent.
Then $(ses^{-1},tet^{-1}) \in \rho$
but $\rho$ is idempotent-separating and so $ses^{-1} = tet^{-1}$.
It follows that $(s,t) \in \mu$.
Thus we have shown that $\rho \subseteq \mu$. 
\end{proof}

We can now explain the connection between the congruence $\mu$ and fundamental inverse semigroups.

\begin{lemma}\label{lem:stormy} Let $S$ be an inverse semigroup.
Then $S$ is fundamental if and only if $\mu$ is the equality relation
\end{lemma}
\begin{proof} Suppose that $S$ is fundamental.
We prove that $\mu$ is the equality relation.
Let $(s,t) \in \mu$.
Then $(st^{-1},tt^{-1}) \in \mu$.
Let $e$ be any idempotent.
Then $(st^{-1})e(st^{-1})^{-1} = tt^{-1}ett^{-1}$.
It follows that 
$$st^{-1}e = e\mathbf{r}(t)st^{-1} = e\mathbf{r}(s)st^{-1} = est^{-1}.$$
We have therefore proved that $st^{-1}$ is central.
By assumption it must be an idempotent.
It follows that $st^{-1} = tt^{-1}$.
Thus $st^{-1}t = t$ but $\mathbf{d}(t) = \mathbf{d}(s)$ and so $s = t$.
To prove the connverse, suppose that $\mu$ is the equality relation.
Let $s$ commute with every idempotent.
Then $(s,ss^{-1}) \in \mu$.
Thus, by assumption, $s = ss^{-1}$ and so $s$ is an idempotent.
\end{proof}

We can easily construct fundamental inverse semigroups.

\begin{lemma}\label{lem:pink} Let $S$ be an inverse semigroup.
Then $S/\mu$ is fundamental.
\end{lemma}
\begin{proof}
Suppose that $\mu (s)$ and $\mu (t)$ are $\mu$-related in $S/\mu$.
Every idempotent in $S/\mu$ is of the form $\mu (e)$ where $e \in E(S)$.
Thus
$$\mu (s)\mu (e)\mu (s)^{-1} = \mu (t)\mu (e) \mu (t)^{-1}$$
so that $\mu (ses^{-1}) = \mu (tet^{-1})$.
But both $ses^{-1}$ and $tet^{-1}$ are idempotents, so that
$ses^{-1} = tet^{-1}$ for every $e \in E(S)$.
Thus $(s,t) \in \mu$.
\end{proof}

The symmetric inverse monoid is constructed from an arbitrary set.
We now show how to construct an inverse semigroup from a meet semilattice.
Let $(E,\leq )$ be a meet semilattice, and denote by $T_{E}$ be the set of all order isomorphisms between the principal order ideals of $E$.  
Clearly, $T_{E}$ is a subset of $\mathcal{I}(E)$.
In fact we have the following.

\begin{lemma}\label{lem:Munn-semigroup} The set $T_{E}$ is an inverse subsemigroup of $\mathcal{I}(E)$
whose semilattice of idempotents is isomorphic to E.
\end{lemma}

The semigroup $T_{E}$ is called the {\em Munn semigroup} of the semilattice $E$. 
We can now construct inverse semigroups having specific semilattices of idempotents.

\begin{theorem}[Munn representation theorem]\label{them:MRT}
Let $S$ be an inverse semigroup.  
Then there is an idempotent-separating homomorphism 
$\delta \colon S \rightarrow T_{\mathsf{E}(S)}$ 
whose image is a wide inverse subsemigroup of $T_{\mathsf{E}(S)}$.
The kernel of $\delta$ is $\mu$.
\end{theorem}
\begin{proof}
For each $s \in S$  
define the function 
$$\delta _{s} \colon \mathbf{d}(s)^{\downarrow} \rightarrow  \mathbf{r}(s)^{\downarrow}$$ 
by $\delta_{s}(e) = ses^{-1}$.
We first show that $\delta_{s}$ is well-defined.  
Let $e \leq  s^{-1}s$.  
Then $ss^{-1}\delta_{s}(e) = \delta_{s}(e)$, 
and so $\delta_{s}(e) \leq  ss^{-1}$.  
To show that $\delta_{s}$ is isotone, let $e \leq f \in (s^{-1}s)^{\downarrow}$.  
Then
$$\delta_{s}(e)\delta_{s}(f) = ses^{-1}sfs^{-1} 
= sefs^{-1} = \delta_{s}(e).$$
Thus $\delta_{s}(e) \leq \delta_{s}(f)$.
Consider now the function
$\delta_{s^{-1}} \colon  (ss^{-1})^{\downarrow} \to (s^{-1}s)^{\downarrow}$.
This is isotone by the argument above.  
For each $e \in  (s^{-1}s)^{\downarrow}$, we have that 
$$\delta_{s^{-1}}(\delta_{s}(e)) = \delta_{s^{-1}}(ses^{-1}) 
= s^{-1}ses^{-1}s = e.$$
Similarly, $\delta_{s}(\delta_{s^{-1}}(f)) = f$ for each $f \in  (ss^{-1})^{\downarrow}$. 
Thus $\delta_{s}$ and $\delta_{s^{-1}}$ are mutually inverse, 
and so $\delta_{s}$ is an order isomorphism.

Define $\delta \colon S \rightarrow T_{\mathsf{E}(S)}$
by $\delta (s) = \delta_{s}$.
To show that $\delta$ is a homomorphism,  
we begin by calculating $\mbox{dom}(\delta_{s} \delta_{t})$
for any $s,t \in S$.
We have that
$$\mbox{dom} (\delta_{s} \delta_{t})  
=  \delta^{-1}_{t} ((s^{-1}s)^{\downarrow} ~ \cap ~ (tt^{-1})^{\downarrow}) 
= \delta^{-1}_{t} ((s^{-1}stt^{-1})^{\downarrow}).$$
But $\delta^{-1}_{t} = \delta_{t^{-1}}$ and so 
$$\mbox{dom} (\delta_{s}  \delta_{t}) 
= ((st)^{-1}st)^{\downarrow} = \mbox{dom}(\delta_{st}).$$
If $e \in \mbox{dom}\,\delta_{st}$ then 
$$\delta_{st}(e) = (st)e(st)^{-1} = s(tet^{-1})s^{-1}
= \delta_{s}(\delta_{t}(e)).$$  
Hence $\delta_{s} \delta_{t} = \delta_{st}$.  

To show that $\delta$ is idempotent-separating, suppose that 
$\delta (e) = \delta (f)$ where $e$ and $f$ are idempotents of $S$.  
Then $\mbox{dom}\,\delta (e) = \mbox{dom}\,\delta (f)$.  
Thus $e = f$.  

The image of $\delta$ is a wide inverse
subsemigroup of $T_{\mathsf{E}(S)}$
because every idempotent in $T_{\mathsf{E}(S)}$ is of the form
$1_{[e]}$ for some $e \in E(S)$,
and $\delta_{e} = 1_{[e]}$.

Suppose that $\delta(s) = \delta (t)$.
Then $(s,t) \in \mathscr{H}$.
It is now immediate from the definition that, precisely, $(s,t) \in \mu$.
\end{proof}

The Munn representation should be contrasted with the Wagner-Preston representation:
the Wagner-Preston was injective whereas the Munn representation has a non-trivial kernel.
Fundamental inverse semigroups arise in the following way.

\begin{theorem}\label{them:white}  
Let $S$ be an inverse semigroup.  
Then $S$ is fundamental if and only if 
$S$ is isomorphic to a wide inverse subsemigroup 
of the Munn semigroup $T_{\mathsf{E}(S)}$.
\end{theorem}
\begin{proof}
Let $S$ be a fundamental inverse semigroup. 
By Theorem~\ref{them:MRT},
there is a homomorphism $\delta \colon S \rightarrow T_{\mathsf{E}(S)}$
such that $\mbox{ker}(\delta) = \mu$.
By assumption, $\mu$ is the equality congruence by Lemma~\ref{lem:stormy}, 
and so $\delta$
is an injective homomorphism.
Thus $S$ is isomorphic to its image in $T_{\mathsf{E}(S)}$,
which is a wide inverse subsemigroup.

Conversely, let $S$ be a wide inverse subsemigroup of a Munn
semigroup $T_{E}$.
Clearly, we can assume that $E = \mathsf{E}(S)$.
We calculate the maximum idempotent-separating congruence of $S$. 
Let $\alpha,\beta \in S$ and suppose that $(\alpha,\beta) \in \mu$ in $S$.
Then $\mbox{dom}(\alpha) = \mbox{dom}(\beta)$.
Let $e \in \mbox{dom}(\alpha)$.
Then $1_{[e]} \in S$, since $S$ is a wide inverse subsemigroup of $T_{\mathsf{E}(S)}$.
By assumption
$\alpha 1_{[e]} \alpha^{-1} = \beta 1_{[e]} \beta^{-1}$.
It is easy to check that 
$1_{[\alpha (e)]} = \alpha 1_{[e]} \alpha^{-1}$
and
$1_{[\beta (e)]} = \beta 1_{[e]}  \beta^{-1}$.
Thus $\alpha (e) = \beta (e)$. 
Hence $\alpha = \beta$, and so $S$ is fundamental.
\end{proof}

The following is a special case of an argument due to Wagner.

\begin{example}{\em Let $(X,\tau)$ be a $T_{0}$-space.
We prove that $\mathcal{I}(X,\tau)$ is fundamental.
Let $f \in \mathcal{I}(X,\tau)$ be a non-idempotent.
Then there is an element $x \in X$ such that $f(x) \neq x$.
Since $X$ is $T_{0}$ there is an open set $U$ such that
either 
$f(x) \in U$ and $x \notin U$
or 
$f(x) \notin U$ and $x \in U$.
In either event, the elements $f1_{U}$ and $1_{U}f$ are not equal,
where $1_{u}$ is an idempotent.
It follows that $\mathcal{I}(X,\tau)$ is fundamental.}
\end {example}

As an example of fundamental inverse semigroups, we can easily construct the fundamental finite Boolean inverse monoids.

\begin{proposition}\label{prop:bordeaux1}
A finite Boolean inverse monoid is fundamental if and only its groupoid of atoms is principal.
\end{proposition}
\begin{proof} 
Let $S$ be fundamental. 
We shall prove that the groupoid of atoms is principal by proving that the local groups are trivial.
Suppose that 
$e \stackrel{a}{\longrightarrow} e$ 
where $e$ is an atom and an idempotent.
We shall prove that $a$ is an idempotent.
Then $a$ is an atom.
Let $f$ be any idempotent.
Then $fa \leq a$.
It follows that $fa = 0$ or $fa = a$.
Suppose that $fa = 0$.
Then $fe = 0$ and so $af = 0$.
Thus $fa = af$.
Suppose now that $fa = a$.
Then $fe = e = fe$ and so $af = a$.
It follows again that $fa = af$.
We have therefore proved that $a$ commutes with every idempotent,
but under our assumption that the inverse semigroup is fundamental,
we deduce that $a = e$. 

Conversely, suppose that $e \stackrel{a}{\longrightarrow} e$, where $e$ is an atomic idempotent, implies that $a = e$.
We shall prove that our semigroup is fundamental.
Let $a$ commute with all idempotents.
We prove that $a$ is an idempotent.
We can write $a = \bigvee_{i=1}^{m} a_{i}$ where the $a_{i}$ are atoms.
We prove that $\mathbf{d}(a_{i}) = \mathbf{r}(a_{i})$ for all $i$ from which the result follows.
Since $\mathbf{r}(a_{j})a = a \mathbf{r}(a_{j})$,
by assumption,
we have that
$a_{j} = \bigvee_{i=1}^{m}a_{i}\mathbf{r}(a_{j})$.
But $a_{i}\mathbf{r}(a_{j}) \leq a_{j}$.
Thus either $a_{i}\mathbf{r}(a_{j}) = 0$ or $a_{i}\mathbf{r}(a_{j}) = a_{j}$
since $a_{j}$ is an atom.
But $a_{i}\mathbf{r}(a_{j}) \leq a_{i}$.
It follows that $a_{i}\mathbf{r}(a_{j}) = a_{i}$ and so $a_{i} = a_{j}$.
Thus $a_{j}\mathbf{r}(a_{j}) = a_{j}$.
Hence $\mathbf{d}(a_{j}) \leq \mathbf{r}(a_{j})$ and so $\mathbf{d}(a_{j}) = \mathbf{r}(a_{j})$, since
both $\mathbf{d}(a_{j})$ and $\mathbf{r}(a_{j})$ are atoms.
By assumption $a_{j}$ is an idempotent.
It follows that $a$ is an idempotent.
\end{proof}

The above result can be used to obtain a more explicit description
of the finite fundamental Boolean inverse monoids.

\begin{theorem}\label{them:finite} Let $S$ be a finite Boolean inverse monoid.
Then it is fundamental if and only if $S$ is isomorphic to a finite direct products $\mathscr{I}_{n_{1}} \times \ldots \times \mathscr{I}_{n_{r}}$.
\end{theorem}
\begin{proof}
It can be checked that the product of two fundamental inverse semigroups
is itself a fundamental inverse semigroup, 
and that the product of two Boolean inverse monoids is again a Boolean inverse monoid.
So, one direction is easy to prove.
Let, now, $S$ be a fundamental finite Boolean inverse monoid.
Then by Proposition~\ref{prop:bordeaux1}, $S$ is isomorphic to a Boolean inverse monoid
of the form $\mathsf{K}(G)$ where $G$ is a principal groupoid.
Now, $G = \bigcup_{i=1}^{i=m}H_{i}$ is a finite disjoint union of connected principal groupoids.
It can be checked that $\mathsf{K}(G) \cong \mathsf{K}(H_{1}) \times \ldots \times \mathsf{K}(H_{m})$.
Let $H$ be any finite connected principal groupoid with set of identities $X$.
Then $\mathsf{K}(H) \cong \mathcal{I}(X)$.  
The result now follows.
\end{proof}

The above result can be specialized to characterize the finite symmetric inverse monoids.
Let $S$ be a Boolean inverse monoid.
A semigroup ideal $I \subseteq S$ is said to be an {\em additive ideal} if $a,b \in I$
and $a \sim b$ implies that $a \vee b \in I$.
Clearly, both $\{0\}$ and $S$ itself are additive ideals.
If these are the only ones we say that $S$ is {\em $0$-simplifying.}

We now have the following theorem which can be derived from Theorem~\ref{them:finite} 
(or see \cite{Lawson2012}).

\begin{theorem} The finite fundamental $0$-simplifying Boolean inverse monoids
are precisely the finite symmetric inverse monoids.
\end{theorem}

The above theorem suggests that the groups of units of fundamental $0$-simplifying Boolean inverse monoids
should be regarded as generalizations of finite symmetric inverse monoids.

\section{Congruence-free inverse semigroups with zero}

An inverse semigroup is said to be {\em congruence-free} if its only congruences are equality and the universal congruence.
In this section, we shall characterize those inverse semigroups with zero which are congruence-free.\footnote{Douglas Munn once remarked to me that this was one of the few instances 
where the theory for inverse semigroups with zero was easier than it was for the one without.}
We begin by ruling out the existence of various kinds of congruence.
An inverse semigroup with zero $S$ is said to be {\em $0$-simple} if the only ideals are $\{0\}$ and $S$.
The following characterization uses Lemma~\ref{lem:jay}
where we describe the $\mathscr{J}$-relation in terms of the $\mathscr{D}$-relation and the natural partial order.

\begin{lemma} Let $S$ be an inverse semigroup with zero.
Then $S$ is $0$-simple if and only if for any two non-zero 
idempotents $e$ and $f$ in $S$ there exists an idempotent $i$ such that
$e\, \mathcal{D}  \,i \leq f$ 
and an idempotent $j$ such that 
$f\, \mathcal{D}  \,j \leq e$. 
\end{lemma}
\begin{proof} Suppose first that $S$ is $0$-simple.
Let $e$ and $f$ be any two non-zero idempotents.
Observe that $e \in SeS$ and $f \in SfS$.
So, both $SeS$ and $SfS$ are not equal to $\{0\}$.
It follows that $S = SeS = SfS$. 
Thus $e \,\mathscr{J}\, f$.
We now use Lemma~\ref{lem:jay} to deduce the result where we have used the fact that the idempotents 
form an order-ideal in an inverse semigroup.
We now prove the converse.
Let $I \neq \{0\}$ be any non-zero ideal of $S$.
Suppose that $I \neq S$.
Let $a \in S\setminus I$.
Observe that $\mathbf{d}(a) \in  S\setminus I$ because if $\mathbf{d}(a) \in I$ then $a = a \mathbf{d}(a) \in I$,
since $I$ is an ideal.
Let $b \in I$ be any non-zero element.
Then $\mathbf{d}(b) \in I$ since $I$ is an ideal.
By assumption, $\mathbf{d}(a)$ and $\mathbf{d}(b)$ are nonzero idempotents.
It follows by the assumption and Lemma~\ref{lem:jay}, that $S\mathbf{d}(a)S = S\mathbf{d}(b)S$.
This implies that $\mathbf{d}(a) \in I$ and so $a \in I$, which is a contradiction.
It follows that, in fact, $I = S$.
\end{proof}

Let $S$ be any inverse semigroup with zero. 
Define
$$(s,t) \in \xi_{S} \Leftrightarrow (\forall a,b \in S)(asb = 0 \Leftrightarrow atb = 0).$$
It is left to the reader to check that this really is a congruence.
We shall denote it by $\xi$ when the semigroup it is defined on is clear.
In the case where $S$ is a meet-semilattice, the above definition simplifies somewhat.
Let $E$ be a meet-semilattice with zero, the operation of which is denoted by concatenation.
Then $(e,f) \in \xi$ if and only ($\forall g \in E$)($eg = 0$ if and only if $fg = 0$).
A congruence $\rho$ is said to be {\em $0$-restricted} if the $\rho$-class containing $0$ is just $0$.
We now have the following characterization of the congruence $\xi$.

\begin{lemma}\label{lem:emily} Let $S$ be an inverse semigroup with zero.
The congruence $\xi$ is the maximum $0$-restricted congruence on $S$.
\end{lemma}
\begin{proof}
Observe, first, that $\xi$ is $0$-restricted.
Suppose that $(s,0) \in \xi$.
By definition, for all $a,b \in S$ we have that $asb = 0$ if and only of $a0b = 0$.
However, if we put $a = ss^{-1}$ and $b = s^{-1}s$ then we deduce that $s = 0$.
Now, let $\rho$ be any $0$-restricted congruence on $S$ and let $s\, \rho \, t$.
Suppose that $asb = 0$.
Then $asb \, \rho  \, atb$. 
Since $\rho$ is $0$-restricted, we have that $atb = 0$.
Thus $asb = 0$ implies that $atb = 0$.
By symmetry, we deduce that $a \, \xi \, b$.
Thus $\rho \subseteq \xi$, as required.
\end{proof}

Our next result shows, amongst other things, the relationship between $\mu$ and $\xi$.

\begin{lemma}\label{lem:fiona} Let $S$ be an inverse semigroup with zero.
\begin{enumerate}

\item $\mu \subseteq \xi$.

\item The congruence $\xi$ restricted to $\mathsf{E}(S)$ is precisely $\xi_{\mathsf{E}(S)}$.

\end{enumerate}
\end{lemma}
\begin{proof} 
(1) Let $s \, \mu \, t$.
We prove that $s \, \xi \, t$.
Suppose that $asb = 0$ for $a,b \in S$.
We shall prove that $atb = 0$.
However, $asb \, \mu \, atb$.
By definition, for all idempotents $e$, we have that
$(asb)e(asb)^{-1} = (atb)e(atb)^{-1}$.
Choose $e = b^{-1}b$.
It follows that 
$\mathbf{r}(asb) = \mathbf{r}(atb)$.
We are given that $asb = 0$ and so $\mathbf{r}(asb) = 0$.
It follows that $\mathbf{r}(atb) = 0$ and so $atb = 0$.
By symmetry, this shows that $s \, \xi \, t$.

(2) Let $e$ and $f$ be idempotents.
To say that $(e,f) \in \xi$ means that for all $a,b \in S$ we have that
$aeb = 0$ if and only if $afb = 0$.
It is clear that  $(e,f) \in \xi_{\mathsf{E}(S)}$.
Suppose, now, that  $(e,f) \in \xi_{\mathsf{E}(S)}$.
We prove that $(e,f) \in \xi$ in $S$.
Let $aeb = 0$.
Then $a^{-1}aebb^{-1} = 0$.
Thus $a^{-1}a bb^{-1} e = 0$ and so, by assumption, $a^{-1}a bb^{-1} f = 0$.
Hence $a^{-1}a f bb^{-1} = 0$ and so $afb = 0$.
The reverse direction is proved similarly.
\end{proof}

We now link $\xi$ being the equality relation on the meet-semilattice $E$
with the property of $E$ being $0$-disjunctive which we introduced 
just before Lemma~\ref{lem:clifford-infinitesimal}. 

\begin{lemma}\label{lem:cathy} Let $E$ be a meet semilattice with zero.
Then  $\xi$ is the equality relation on $E$ if and only if $E$ is $0$-disjunctive.
\end{lemma}
\begin{proof} Suppose first that $\xi$ is the equality relation on $E$.
We prove that $E$ is $0$-disjunctive.
Suppose that $0 < f < e$
and assume that there does not exist any $0 \neq g \leq e$ such that $fg = 0$.
If $ei = 0$ then clearly $fi = 0$.
Suppose that $fi = 0$.
Then $fi = fei = 0$ and so $f(ei) = 0$.
Clearly, $ei  \leq e$ and so by our assumption above, we must have $ei = 0$.
It follows that $(e,f) \in \xi$.
But this implies that $e = f$, which is a contradiction.

We assume that $E$ is $0$-disjunctive and prove that $\xi$ is the equality relation.
Suppose that $(e,f) \in \xi$, where $e$ and $f$ are both non-zero.
Then  $e\, \xi \, ef$ and so  $ef \neq 0$ since $\xi$ is $0$-restricted. 
Suppose that $ef \neq e$.
Then $0 < ef < e$.
Then, by assumption, there exists $0 \neq g \leq e$ such that $(ef)g = 0$.
But, clearly, $g = eg \neq 0$.
However, $e\, \xi \, ef$ and so we have a contradiction. 
It follows that $ef = e$.
Similarly, $ef = f$ and so $e = f$, as required.
\end{proof}

We are nearly at our goal.

\begin{lemma}\label{lem:liz} Let $S$ be an inverse semigroup with zero.
Then $\xi$ is the equality relation if and only if $\mathsf{E}(S)$ is $0$-disjunctive and $S$ is fundamental.
\end{lemma}
\begin{proof} Suppose first that $\xi$ is the equality relation. 
Then $\mathcal{E}(S)$ is $0$-disjunctive
by Lemma~\ref{lem:fiona} and Lemma~\ref{lem:cathy},
and it is fundamental by Lemma~\ref{lem:fiona}.
To prove the converse, suppose that $\mathsf{E}(S)$ is $0$-disjunctive and $S$ is fundamental.
Then by Lemma~\ref{lem:fiona} and Lemma~\ref{lem:cathy},  
$\xi$ restricted to $\mathsf{E}(S)$ is the equality relation and so $\xi$ is idempotent-separating.
It follows by Lemma~\ref{lem:largest} that $\xi \subseteq \mu$.
But $S$ is fundamental and so $\mu$ is the equality relation thus $\xi$ is the equality relation.
\end{proof}

We may now state the characterization of congruence-free inverse semigroups with zero.

\begin{theorem}[Congruence-free inverse semigroups with zero] 
An inverse semigroup with zero $S$ is congruence-free if and only if
$S$ is fundamental, $0$-simple and $\mathsf{E}(S)$ is $0$-disjunctive.
\end{theorem}
\begin{proof}
Suppose that $S$ is congruence-free. 
Then $\mu$ is equality, 
there are no non-trivial ideals and $\xi$ is equality.
Thus $S$ is fundamental, $0$-simple and $\mathsf{E}(S)$ is $0$-disjunctive by Lemma~\ref{lem:cathy}

To prove the converse, suppose that $S$ is fundamental, $0$-simple and $\mathsf{E}(S)$ is $0$-disjunctive.
Let $\rho$ be a congruence on $S$ which is not the universal relation.
Then $\rho (0)$ is an ideal which is not $S$.
Thus this ideal must be equal to $\{0\}$.
It follows that $\rho$ is a $0$-restricted congruence and so $\rho \subseteq \xi$ by Lemma~\ref{lem:emily}. 
Now we use the fact that the semigroup is fundamental together with Lemma~\ref{lem:liz}, 
to deduce that $\xi$ is the equality congruence and so $\rho$ is the equality congruence.
\end{proof}

\section{Further reading and examples}

My interest in inverse semigroups developed as a result of John Fountain's undergraduate course in semigroups at the University of York, UK.
It was deepened by attending the seminar at Oxford organized by Peter Cameron, Wilfrid Hodges, Angus Macintyre, and Peter Neumann,
in which partial automorphisms and the work of Fra\"{\i}ss\'e
played a central role.
The connections between Fra\"{\i}ss\'e's work and inverse semigroups
was clarified in \cite{TSBB} though based on the work of Benda to be found \cite{Benda}.
Partial automorphisms also play a significant role in Fra\"{\i}ss\'e's approach to the study of relational structures \cite{Fraisse}.
The work of Ehresmann on pseudogroups \cite{E} convinced me that inverse semigroups really
were worth studying.
I was, however, bugged by the question
of whether there were `natural' examples of inverse semigroups.
But what makes an area of mathematics `natural'?
I can think of at least two answers to this question:
the existence of deep problems and the proliferation of good examples.
When I embarked on writing my book \cite{Lawson},
I was primarily motivated by finding good examples.
The work of Kellendonk \cite{K1995, K1997a, K1997b}
and
Girard \cite{GLR}, as mediated by Peter Hines \cite{Hines},
provided such good examples.
Since that time, it has become clear that partial bijections, and therefore inverse semigroups, arise
naturally in many different parts of mathematics.
There are also deep questions such as whether every finite inverse monoid has a finite $F$-inverse cover \cite{ABO}
or what can be said about finite semigroups with commuting idempotents \cite{Ash}.
In this section, I shall concentrate on sources of good examples, since
these look like promising growth points for the theory.

If there is one area of mathematics which makes essential use of 
inverse semigroups, it is in the theory of $C^{\ast}$-algebras.
There are so many books and papers that are relevant here that I will only mention a few to get you started.
The connection between inverse semigroups and $C^{\ast}$-algebras was first explored in \cite{K} and \cite{R}.
A rather more recent book on this topic is \cite{Paterson} which will repay reading.
There are many, many papers on this topic.
A good place to start is \cite{Exel} and \cite{FMY}
but it is worth checking out \cite[page 45]{Putnam}
which is on the special case of Cantor minimal systems.
Let me, here, give one example of the theory of $C^{\ast}$-algebras leading to new semigroup theory.
A monoid is said to be {\em finitely aligned} if the intersection of any two principal right ideals
is a finitely generated right ideal.
This concept was first defined within the theory of $C^{\ast}$-algebras in \cite{RS} and made extensive use of in \cite{Spielberg2014, Spielberg2018}.
However, it seems to have been first defined within semigroup theory in \cite{VARG}, quite early on,
but was not picked up by workers in inverse semigroup theory;
this is but one illustration of the importance of mathematicians in different fields talking to one another.
See also the more recent \cite{CG}.
It would be most naturally applied in \cite{CP}.

The connection between inverse semigroups and $C^{\ast}$-algebras has come to be mediated by what are now termed
{\em Steinberg algebras}; see \cite{Steinberg2010, SS}. 
Even where no direct connection between inverse semigroups and $C^{\ast}$-algebras
appears to exist, the theory of $C^{\ast}$-algebras has sometimes led to new developments in the theory
of inverse semigroups. A good example is Steinberg's notion of strong Morita equivalence of inverse semigroups
motivated by the corresponding theory for $C^{\ast}$-algebras \cite{Steinberg2009}.

The connections between inverse semigroups and quasi-crystals were first spelt out by
Kellendonk \cite{K1995, K1997a, K1997b}.
The inverse semigroup perspective on the topological groupoids that Kellendonk
constructs were first clarified by Lenz \cite{Lenz} and then developed in \cite{LMS}.
Aspects of this connection were described in \cite{LL}.
This led to what we might term `non-commutative Stone duality'
which connects, in particular, Boolean inverse monoids with a class of \'etale groupoids.
For a survey of non-commutative Stone duality, see \cite{Lawson2022}.

I once asked the late Pieter Hofstra, an expert in topos theory, why I kept seeing inverse semigroups everyone.
I don't know whether Pieter was being polite, but he explained that it was because \'etendues were everywhere,
where an \'etendue is a particular kind of topos.
This seems to me a particularly fruitful area of research.
We refer the reader to \cite{FH2021, FH2022}.

If the theory of inverse semigroups is to be more fully developed,
then there will be a need to study invariants associated with inverse semigroups.
Motivated by the Banach-Tarski Paradox \cite{CSGH} and the theory of $K_{0}$-groups \cite{Goodearl},
the type monoid was introduced in \cite{KLLR, Wallis} and its theory developed by Wehrung \cite{Wehrung}.
Type monoids are commutative refinement monoids and provide the first example of invariants of inverse semigroups.
Another source of examples of invariants of inverse semigroups
are cohomology theories.
These have barely been developed;
see the work of Lausch \cite{L} and, particularly, Loganathan \cite{Log}.
An application of the extension theory of inverse semigroups in the case of Von Neumann algebras can be found in \cite{DFP}.
MV-algebras form a natural class of invariants for the class
of factorizable Boolean inverse monoids $S$ in which $S/\mathscr{J}$ is a lattice.
MV-algebras are another generalization of Boolean algebras
but it  was proved in \cite{LS} that all countable MV-algebras arise in this way
and this was generalized to all MV-algebras by Wehrung \cite{Wehrung} using very different methods.
Thus every MV-algebra arises from a particular kind of Boolean inverse monoid.
The theory of MV-algebras goes back to Tarski's book \cite{Tarski},
which is also noteworthy in that partial automorphisms of structures
are singled out in \cite[Theorem 11.6]{Tarski}.
MV-algebras are studied in their own right in \cite{CDM, Mundici, M3}
but their applications to the study of certain kinds of Boolean inverse monoids
has so far not been developed.

I have indicated that certain kinds of groups can be constructed from inverse semigroups.
See \cite{Matui12, Matui13} for how groups arise as so-called `topological full groups' of certain kinds of 
groupoids; under non-commutative Stone duality, this means that the groups arise as the groups of units
of certain kinds of Boolean inverse monoids.
The Thompson groups form an important class of such groups.
They, too, can be constructed from inverse monoids \cite{JL, LV, LSV}
generalizing some initial work in \cite{Birget}.
Groups with a different flavour arise from what Dehornoy calls `geometry monoids'.
This is explained in his book \cite{PD} but the explicit connection with inverse monoids is made in \cite{Lawson2006}.

If I were to single out one class of inverse semigroups that justifies the field
it would be the polycyclic inverse monoids introduced by Nivat and Perrot \cite{NP}.
They were generalized by Perrot \cite{Perrot, Perrot1972}.
As I pointed out in \cite{Lawson2008}, Perrot's thesis contains the first definition 
of self-similar group actions \cite{Nek2005} and was notivated entirely by inverse semigroup theory.
My paper was subsequently generalized by Alistair Wallis \cite{Wallis};
see also \cite{LW2015, LW2017}.
The polycyclic inverse monoids are intimately related to the classical Thompson groups \cite{Lawson2007, Lawson2007b, Lawson2021}.
The representation theory (by means of partial bijections) of the polycyclic inverse monoids
is the subject of \cite{Jones, JL1, Lawson2009} and was motivated by \cite{BJ} and by \cite{Kawamura}, 
which on the face of it have nothing to do with the polycyclic inverse monoids.
The polycyclic inverse monoids arise from the free monoid.
Analogous inverse monoids arise from free categories \cite{AH, JL}.
There is a natural connection between such inverse semigroups and Leavitt path algebras \cite{AAM}.
This is part of the general topic of studying algebras of various kinds associated with inverse semigroups.
This goes back to pioneering work of Douglas Munn \cite{Munn} and it reaches its apotheosis in the applications
of inverse semigroups to the theory of $C^{\ast}$-algebras mentioned above.
This is a subject in its own right, but we mention in passing \cite{DM} and \cite{LV2021}, which are relevant.

To conclude, I have not said anything about the classical theory of inverse semigroups.
Here, I would simply highlight the papers \cite{Ash1979, M0, Meakin} with the suggestion that there is more to do here.
I have also said nothing about free inverse semigroups.
These are introduced in my book \cite[Chapter 6]{Lawson} but a much more recent paper on their structure is \cite{Gray}.
In passing, I would like to mention that free inverse monoids can also be constructed using 
the doubly pointed pattern classes of Kellendonk applied to the graph of the free group regadred as a `dendritic tiling' \cite{K1997a}.
In fact, it was this example which convinced me that Kellendonk was on to something.
The study of free inverse monoids leads naturally to the study of presentations of inverse semigroups.
A recent paper on this topic is \cite{Gray2020}.


\end{document}